\documentclass[11pt,reqno,a4paper, english]{amsart}
\usepackage[top=3.5cm,bottom=3cm,left=2.5cm,right=2.5cm]{geometry}
\usepackage[mathscr]{euscript}

\usepackage{amsmath, mathtools, amsthm, amsfonts, amssymb,enumerate}
\usepackage[dvipsnames]{xcolor}
\usepackage{tikz}
\usepackage{microtype}
\usepackage{hyperref}
\hypersetup{colorlinks={true},linkcolor={blue},citecolor=blue}
\usepackage{comment}

\definecolor{bole}{rgb}{0.47, 0.27, 0.23}
\definecolor{amber}{rgb}{1.0, 0.75, 0.0}

\numberwithin{equation}{section}

\makeatletter
\newenvironment{figurehere}
{\def\@captype{figure}}
{}
\makeatother

\makeatletter
\def\subsubsection{\@startsection{subsubsection}{4}%
  \z@\z@{-\fontdimen2\font}%
  {\newline\normalfont\bfseries}}
\makeatother

\theoremstyle{plain}
\newtheorem{theorem}{Theorem}[section]
\newtheorem{prop}[theorem]{Proposition}

\newtheorem{lem}[theorem]{Lemma}
\newtheorem{lemma}[theorem]{Lemma}

\theoremstyle{definition}
\newtheorem{definition}[theorem]{Definition}

\theoremstyle{remark}
\newtheorem{remark}[theorem]{Remark}

\DeclareMathOperator{\supp}{supp}

\DeclareMathOperator{\e}{e}

\DeclareMathOperator{\hi}{hi}
\DeclareMathOperator{\lo}{lo}

\renewcommand{\epsilon}{\varepsilon}

\newcommand{\N}{\mathbb{N}}
\newcommand{\Z}{\mathbb{Z}}

\newcommand{\R}{\mathbb{R}}
\newcommand{\C}{\mathbb{C}}
\newcommand{\T}{\mathbb{T}}

\newcommand{\abs}[1]{\vert #1 \vert}

\newcommand{\norm}[1]{\| #1 \|}

\newcommand{\normbig}[1]{\biggl\| #1 \biggr\|}
\newcommand{\parent}[1]{\bigl(#1\bigr)}
\newcommand{\parentbig}[1]{\biggl(#1\biggr)}
\newcommand{\set}[1]{\bigl\{#1\mathclose{}\bigr\}}

\newcommand{\japbrak}[1]{\langle#1\rangle}

\newcommand{\interval}[4]{\mathopen{#1}#2\mathclose{}\mathpunct{},#3\mathclose{#4}}
\newcommand{\intervaloo}[2]{\interval{(}{#1}{#2}{)}}
\newcommand{\intervaloc}[2]{\interval{(}{#1}{#2}{]}}

\newcommand{\intervalcc}[2]{\interval{[}{#1}{#2}{]}}

\newcommand\dd{\,{\mathrm d}}
\newcommand{\un}{\mathbf{1}}

\newcommand{\A}{\mathcal{A}}

\begin{document}


\title[Almost sure LWP for cubic Schrödinger half-wave]{Refined probabilistic local well-posedness for a cubic Schrödinger half-wave equation}

\author{Nicolas Camps, Louise Gassot, Slim Ibrahim}
\address{Universit\'e Paris-Saclay, Laboratoire de mathématiques d'Orsay, UMR 8628 du CNRS, B\^atiment 307, 91405 Orsay Cedex, France}
\email{nicolas.camps@universite-paris-saclay.fr}
\address{
Departement
Mathematik und Informatik, Universität Basel, Spiegelgasse 1, 4051 Basel, Schweiz}
\email{louise.gassot@normalesup.org}

\address{Department of Mathematics and Statistics
University of Victoria, 3800 Finnerty Road, Victoria
BC V8P 5C2, Canada\newline
Pacific Institute for Mathematical Sciences, 4176-2207 Main Mall, Vancouver, BC V6T 1Z4,
Canada}
\email{ibrahims@uvic.ca}



\date{\today}

\begin{abstract}
We obtain probabilistic local well-posedness in quasilinear regimes for the Schrödinger half-wave equation with a cubic nonlinearity. We need to use a refined ansatz because of the lack of probabilistic smoothing in the Picard's iterations, which is due to the high-low-low nonlinear interactions. The proof is an adaptation of the method of Bringmann on the derivative nonlinear wave equation~\cite{bringmann2021-ansatz} to Schrödinger-type equations. In addition, we discuss ill-posedness results for this equation. 
\end{abstract}

\ \vskip -1cm  \hrule \vskip 1cm \vspace{-8pt}
 \maketitle 
{ \textwidth=4cm \hrule}

\maketitle
\tableofcontents

\section{Introduction}
We consider the Cauchy problem at low regularity for the cubic Schrödinger half-wave equation on $\R^2$
\begin{equation}
\label{eq:NLSHW}
    \tag{NLS-HW}
    \begin{cases}
    i\partial_tu + \parent{\partial_{xx}- \abs{D_y}}u = \mu \abs{u}^{2}u\,,\quad (t,x,y)\in \R^{1+2}\,,\quad \widehat{\abs{D_y}u}(\xi,\eta) = \abs{\eta} \widehat{u}(\xi,\eta)\,,\\
    u_{t=0} = u_0\,,
    \end{cases}
\end{equation}
with $\mu\in\R^*$. The study of this equation is motivated by mathematical interests, as a model for understanding weakly dispersive equations that are not completely integrable, but for which some global solutions have unbounded Sobolev norms.


~\\
\noindent{\bf Background on the Schrödinger half-wave equation } This model was first introduced by Xu~\cite{Xu2017unbounded} to evidence weak turbulence mechanisms.  In the defocusing case $(\mu>0)$, global existence and modified scattering is obtained for a class of sufficiently smooth and decaying small initial data on the wave guide $\R_x\times\T_y$. In addition, the author shows that the limiting effective dynamics is governed by the Szeg\H{o} equation on the torus. As a consequence of~\cite{GerardGrellier2012,GerardGrellier2017}, this analysis provides arbitrarily small initial data such that for every $s>\frac 12$ and $N\geq 1$, the solution $u$ exhibits weak turbulence in the sense that
\[
\limsup_{t\to\infty} \frac{\|u(t)\|_{L^2_xH^s_y}}{\log(t)^N}=\infty\,,\quad
 \liminf_{t\to\infty} \|u(t)\|_{L^2_xH^s_y}<\infty\,.
\]
Subsequently, Bahri, Ibrahim and Kikuchi consider in~\cite{BahriIbrahimKikuchi2020remarks, BahriIbrahimKikuchi2021transverse} the focusing case $\mu<0$  and study traveling wave solutions on the wave guide $\R_x\times\T_y$. They obtain orbital stability and transverse instability results, subject to the condition of a good Cauchy theory in the energy space. Unfortunately, such a Cauchy theory is yet to be addressed, and not much is known about the global existence of smooth solutions in Sobolev spaces. 

The main concern in this paper is to address the local Cauchy problem at low regularity, for initial data in a \emph{statistical} ensemble. Before presenting our results, let us recall that~\eqref{eq:NLSHW} is a Hamiltonian system with a conserved energy 
\[
H(u) = \frac{1}{2}\int_{\R^2} \left|\partial_xu\right|^2 + \left|\abs{D_y}^\frac{1}{2}u\right|^2 \dd x \dd y+ \frac{\mu}{4}\int_{\R^2}|u|^4\dd x\dd y\,.
\]
The $L^2$ mass is also formally conserved by the flow, and the solutions are left invariant by the scaling symmetry
\begin{equation}
\label{eq:scaling}
    u\mapsto u_{\lambda}(t,x,y)=\lambda u(\lambda^2 t,\lambda x,\lambda^2 y)\,.
\end{equation}
In light of the conservation laws and of the scaling invariance, the relevant regularity spaces for this equation are the following anisotropic Sobolev spaces $\mathcal{H}^s$ defined by
\[
\mathcal{H}^s:=L^2_xH^{s}_y\cap H^{2s}_xL^2_y\,,\quad \dot{\mathcal{H}}^s:=L^2_x\dot{H}^{s}_y\cap \dot{H}^{2s}_xL^2_y\,.
\]
The scaling~\eqref{eq:scaling} leaves the $\dot{\mathcal{H}}^\frac{1}{4}$-norm invariant. As a consequence, when $0<s<\frac{1}{4}$, a \emph{low-to-high frequency cascade} occurs and there is a short-time inflation of the $\dot{\mathcal{H}}^s$-norm for the regularized solutions~\cite{Kato2021}. We make precise the type of ill-posedness by describing a {\it pathological set} in Appendix~\ref{sec:ill-posedness}. When $\frac{1}{4}<s$, the equation is scaling-subcritical. Yet, semilinear local well-posedness is only known in $\mathcal{H}^s$ when $\frac{1}{2}<s$. Specifically, semilinear well-posedness is obtained in~\cite{BahriIbrahimKikuchi2020remarks} from Strichartz estimates with a derivative loss.    
Furthermore, we prove in Appendix~\ref{sec:ill-posedness} that the flow map cannot be of class $\mathcal{C}^3$ when $\frac{1}{4}<s<\frac{1}{2}$. To do so, we consider a one-parameter family of traveling wave profiles for the one-dimensional Szeg\H{o} equation to invalidate some relevant Strichartz estimates. It is therefore not possible to run a contraction mapping argument when $\frac{1}{4}<s<\frac{1}{2}$, since otherwise the flow-map would be analytical. The state-of-the-art Cauchy theory results for~\eqref{eq:NLSHW} are summarized in the following diagram:

\begin{figurehere}
\begin{center}
\begin{tikzpicture}
\draw[line width=3, bole] (1,0)--(5,0);
\draw[line width=3, amber] (5,0)--(10,0);
\draw[bole] (3,-0.1) node[below] {Norm-inflation~\cite{Kato2021}};
\draw[orange] (7,-0.1) node[below] {Flow map is not $\mathcal{C}^3$~\cite{bgt-05}};
\draw[Cerulean] (11.5,-0.1) node[below] {Local well-posedness~\cite{BahriIbrahimKikuchi2020remarks}};
\draw[->, line width=3, Cerulean] (9,0)--(14,0);
\draw (14.5,0)  node[above]{$\mathcal{H}^s$};
\draw (1,0) node{{\color{bole} \bf{I}}} node[above=2]{ $L^2$};
\draw (5,0) node{{\color{amber} \bf{I}}} node[above=2]{ $\mathcal{H}^\frac{1}{4}$};
\draw (9,0) node{{\bf I}} node[above=2]{$\mathcal{H}^\frac{1}{2}$};
\draw (-0,0)--(14,0);
\end{tikzpicture}
\end{center}
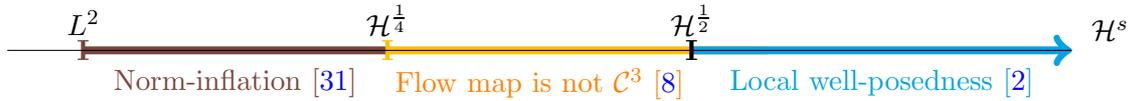
\caption{Cauchy theory for equation~\eqref{eq:NLSHW} in the scale of Sobolev spaces. 
 }
\end{figurehere}

~\\
{\bf Main result}
As mentioned above, it turns out to be quite challenging to prove semilinear local well-posedness in the energy space, or to run a quasi linear scheme in $\mathcal{H}^s$ when $\frac{1}{4}<s<\frac{1}{2}$. Instead, our goal is to study the \emph{probabilistic} local well-posedness of equation~\eqref{eq:NLSHW} in this range of regularities.  Given $f_0\in\mathcal{H}^s(\R^2)$ and a sequence of independent normalized Gaussian variables $(g_k(\omega))_{k\in\Z}$ on a probability space $(\Omega,\mathcal{A},\mathbb{P})$, we define a random variable 
\begin{equation*}
    \omega\in\Omega\mapsto f_0^\omega := \sum_{k\in\Z} g_{k}(\omega)P_{1,k}f_0\,,
\end{equation*}
where $P_{1,k}$ is the partial Fourier projector (in the $y$-variable) on an interval of unit length centered around $k\in\Z$. We refer to~\eqref{eq:randomization} for the precise definition of the relevant anisotropic Wiener randomization procedure we employ, and to Figure~\ref{fig:Pnk}. The considered probability measure on $\mathcal{H}^s$ is the induced measure by the above random variable.

However, due to the lack of dispersion and to the absence of regularizing features (such as bilinear estimates or local energy decay), Picard's iterations have the same regularity as the initial data and the standard probabilistic method, which is discussed below, fails. Instead, we need to consider a refined probabilistic ansatz adapted from the probabilistic quasilinear scheme developed by Bringmann~\cite{bringmann2021-ansatz} to prove probabilistic well-posedness for a derivative wave equation. For $T>0$ and some $s$ and $\sigma$ in $\R$, we denote 
\begin{equation}
    \label{eq:XT0}
X_{T_0}^{s,\sigma} := \, \mathcal{C}_t\parent{\intervalcc{-T_0}{T_0}\, ;\, \mathcal{H}^s(\R^2)} \cap L_t^8\left(\intervalcc{-T_0}{T_0}\, ;\,  L_x^4W_y^{\sigma,\infty}(\R^2)\right)\,.
\end{equation}
We also denote the truncated initial data as
\begin{equation*}
    P_{\leq n}f_0^\omega := \sum_{|k|\leq n} g_{k}(\omega)P_{1,k}f_0\,.
\end{equation*}
Our main result reads as follows. 
\begin{theorem}[Probabilistic local well-posedness]
\label{theo:main}
Let $s\in(13/28, 1/2]$, some suitable $0<\sigma<s$ and $f_0\in\mathcal{H}^s$.
There exist $T_0>0$ and a sequence $(u_n)_{n\geq 1}\in \mathcal{C}([-T_0,T_0]\, ;\, \mathcal{H}^\infty)^\N$ converging in expectation to a limiting object denoted $u$
\begin{equation}
    \label{eq:convergence}
    \underset{n\to\infty}{\lim}\,\mathbb{E}\left[\left\| u_n - u \right\|_{X_{T_0}^{s,\sigma}}^2 \right] = 0\,,
\end{equation}
in such a way that, almost-surely in $\omega\in\Omega$, there exists $T^\omega>0$ such that for all $n\in\N$, $u_n$ and $u$ exist in $\mathcal{C}([-T^\omega, T^\omega] ; \mathcal{H}^s)$ and
solve~\eqref{eq:NLSHW} with initial data $P_{\leq n}f_0^\omega$ and $f_0^\omega$, respectively.
\end{theorem}

\begin{remark}[Time of existence]
Theorem 1.1 provides a uniform time of existence $T^\omega>0$ for the smooth solutions $u_n$ initiated from regularized initial data  $(P_{\leq n}f_0^\omega)_n$ in the \emph{statistical} ensemble. This is not a consequence of the known local well-posedness result from~\cite{BahriIbrahimKikuchi2020remarks}. Then, the main claim of Theorem~1.1 is the  convergence of $(u_n)_n$ to a strong solution $u$ in $\mathcal{H}^s$, on the time interval $[-T^\omega,T^\omega]$. 
\end{remark}

\begin{remark}[Regularizing sequence]
Note that the regularizing sequence is general and not dyadic. Actually, we first prove the result for dyadic frequencies $N$ and deduce the convergence for general approximations by adapting an argument from~\cite{SunTzvetkov2021derivative}.
\end{remark}
\begin{remark}[Other models]
Theorem 1.1 is an adaptation to Schrödinger-type equations of the quasilinear scheme from~\cite{bringmann2021-ansatz}, and allows addressing in the same way other weakly dispersive models. For instance, the case of the one-dimensional half-wave equation and of the Szeg\H{o} equation are both contained in our analysis by simply forgetting about the variable $x$ everywhere. 
\end{remark}


\begin{remark}[Nonlinearity]
We decided to only consider the cubic nonlinearity on $\R^2$ for simplicity. But up to some technicalities, other nonlinearities $p\geq5$ and other geometries (e.g. on the wave guide $\R\times\T$) would work.
\end{remark}

\begin{remark}[Sobolev threshold]
The threshold $s>s_0$ with $s_0=13/28$  is a convenient fractional approximation of the exponent $s_0$ we will actually get in~\eqref{eq:opti}. This is by no means optimal, and it would be interesting to go all the way down to $s_0=1/4$ to cover the whole quasilinear range of exponents.
\end{remark}
~\\
\paragraph{\textbf{Background}} This work is part of the study of nonlinear weakly dispersive equations, in the presence of randomness.
~\\

\emph{Weakly dispersive equations}. There has been a rich activity on the qualitative behavior of solutions to nonlinear evolution equations with weak dispersion. Examples are the half-wave equation~\cite{Pocovnicu2011,GerardLenzmannPocovnicuRaphael2018} and other fractional Schrödinger equations~\cite{FrankLenzmann2013,GuoPuHuang2015,SunTzvetkov2021derivative}, the Szeg\H{o} equation~\cite{GerardGrellier2010,GerardGrellier2012invariant}, the Schrödinger equation on compact manifold~\cite{bgt-05}, the Schrödinger equation on the Heisenberg group~\cite{Gerard2006} for which there is a lack of dispersion in one direction, the Kadomtsev-Petviashvili equations~\cite{hadac2009} and many other models. It is in general quite challenging to solve the Cauchy problem for these equations at low regularity. The reason being that in such contexts, dispersive estimates come with a derivative loss. In some special cases where the Hamiltonian equation is integrable in the sense that it has a Lax pair structure providing infinite conservation laws, methods from integrable systems are substituted for the analytical perturbative approach, to study both Cauchy theory and the long-time dynamics.

A first illustrative example one can think of is the Szeg\H{o} equation, which is known to be globally well-posed in $H^{\frac 12}_+(\T)$ since~\cite{GerardGrellier2010}. Using the Lax pair structure, the flow map was extended to $\operatorname{BMO}$ in~\cite{GerardKoch2016}, then the Cauchy problem was recently shown to be globally well-posed even on $L^2_+(\T)$ in~\cite{GerardPushnitski2022} also by using integrable techniques.
Note that the cubic Szeg\H{o} equation is invariant by a $L^4_+(\T)\hookrightarrow H^{\frac 14}_+(\T)$-critical scaling.
Another model is the derivative Schrödinger equation on the line, which is $L^2$-critical with respect to the scaling, but the flow map fails to be uniformly continuous in $H^s(\R)$ when $s<\frac 12$~\cite{BiagioniLinares2001,Takaoka2001}. Nevertheless,~\cite{HarropKilipNtekoumeVisan2022} recently proved from integrable techniques that the equation is globally well-posed in $L^2(\R)$. Previously, integrable techniques were combined with concentration-compactness argument in~\cite{BahouriPerelman2020} to prove global well-posedness in $H^{\frac 12}(\R)$.
A third model is the Schrödinger equation on the Heisenberg group, for which local well-posedness is only known to hold in adapted Sobolev spaces $H^s(\mathbb{H}^1)$ for $s>2$, but is $H^{\frac 12}$ scaling-invariant, moreover, the flow map is known not to be of class $\mathcal{C}^3$ for $\frac{1}{2}<s<2$. The lack of dispersion comes from a special direction, the vertical direction, where the equation is  a transport equation rather than a dispersive one. 

For the Grushin Laplacian (a simplified version of the Heisenberg Laplacian), a probabilistic approach was attempted in~\cite{schroheis3}, proving almost-sure local well-posedness according to a specific randomization procedure which penalizes the direction where there is no dispersion, and gain derivatives in $L^p$ spaces. Finally, we stress out that it is sometimes possible to close the derivative gap (namely $1/4<s<1/2$ in our situation) \emph{deterministically} by considering more general function spaces than the Sobolev spaces, such as Fourier-Lebesgue spaces, see e.g.~\cite{Grunrock2005} for the derivative NLS equation,  and~\cite{Grunrock2011-NLW} for the wave equation.




~\\

\emph{Probabilistic Cauchy theory and its limitations}. The so-called probabilistic Cauchy theory was initiated by Bourgain~\cite{bourgain96-gibbs} and Burq, Tzvetkov~\cite{burq-tzvetkov-2008I,burq-tzvetkov-2008II}, and can be sketched as follows. When $s<s_c$ is below a critical threshold $s_c$ so that instabilities are known to occur, a general initial data in $\mathcal{H}^s$ may have better integrability properties than expected by the Sobolev embedding. In contexts when the dispersion is strong enough to give some local energy decay or bilinear estimates, or for wave-type equations where the Duhamel formula gains one derivative, one can exploit such an enhanced integrability property to prove that the Picard's iterations are smoother than the linear evolution of the initial data. A rigorous analysis of this observation makes it possible to evidence strong solutions initiated from \emph{statistical} initial data, according to a non-degenerate probability measure which charges any open set in $\mathcal{H}^s$.


A \emph{randomized} initial data is typically defined from a given function $\phi\in\mathcal{H}^s$ and its unit-scale frequency decomposition $(\phi_n^\omega)_n$ in the frequency space (or according to a spectral resolution of the Laplace operator in compact settings), where each mode is decoupled by a sequence $(g_n(\omega))_n$ of normalized independent Gaussian variables, defined on a probability space $(\Omega,\mathcal{F},\mathbb{P})$: 
\[
\omega\in\Omega\mapsto\phi^\omega \sim \sum_{n} g_n(\omega)\phi_n\,,\quad \text{where}\ \phi\sim\sum_n\phi_n\in\mathcal{H}^s\,.
\]
Then, for \emph{many} initial data $\phi^\omega$ in a statistical ensemble $\Sigma\subset \mathcal{H}^s$ which has full measure, one expects to observe a \emph{nonlinear smoothing effect} for the recentered solution around the linear evolution thanks to the combination of space-time oscillations (dispersion) and probabilistic oscillations (randomization). Namely, the goal is to show that for some $\nu>s_c$,
\[
v(t) := u(t) - \e^{it\Delta}\phi^\omega \in \mathcal{C}(\intervalcc{-T}{T};\mathcal{H}^\nu) \quad \text{for all}\ \phi^\omega\in\Sigma\,.
\]
The above recentered solution $v$ is obtained from a fixed point argument at subcritical regularities in $\mathcal{H}^\nu$ and solves the original equation perturbed by stochastic terms stemming from the linear evolution $\e^{it\Delta}\phi^\omega$. We refer to~\cite{burq-tzvetkov-2008I,burq-tzvetkov-2008II} for an introduction to the probabilistic Cauchy theory for dispersive equations.

Nevertheless, to observe the aforementioned probabilistic smoothing effect, we need to exploit dispersive properties of the equations to gain decay without trading regularity. In equation~\eqref{eq:NLSHW}, however, there is no dispersion in the $y$-direction. Therefore, in the low $\widehat{x}$-frequency regimes, the Strichartz estimates come with a derivative loss, so that we have neither usable bilinear estimate, nor local smoothing estimates at our disposal. Besides, the second Picard iteration of the randomized initial data does not have a better regularity than the initial data. To see this, we consider \emph{high-low-low} type interactions in high $y$-frequencies $|\eta|\gg 1$, for an initial data $\phi^\omega$ projected at low $x$-frequencies $|\xi|\lesssim 1$. For simplicity, we assume that there is no dependence in the variable $x$, then the {\it high-low-low} interactions take on the form
\begin{equation*}
\label{eq:transport-picard}
     \int_0^t\e^{-i(t-\tau)\abs{D_y}}P_{|\eta|\gg1}\e^{i\tau|D_y|}\phi^\omega \overline{P_{\abs{\eta}\leq1}\e^{i\tau|D_y|}\phi^\omega}P_{\abs{\eta}\leq1}\e^{i\tau|D_y|}\phi^\omega\dd\tau.
\end{equation*}
Such interactions are basically transported by the half-wave linear flow, so that every derivative of the second Picard's iteration can fall onto the first linear term. Hence, one can only handle $s$ derivatives instead of the desired $\nu$ derivatives. Similarly, Oh~\cite{Oh2011remarks} considered the Szeg\H{o} equation on the circle and proved that  the first nontrivial Picard's iterate does not gain regularity compared to the initial data.

To undertake this type of issues,  Bringmann~\cite{bringmann2021-ansatz} developed a refined probabilistic ansatz in a quasilinear setting thanks to paracontrolled calculus. In the present work, we adapt the strategy from~\cite{bringmann2021-ansatz} to prove almost-sure local well-posedness below the energy space, in the quasilinear regime. We conclude this paragraph by mentioning that the paracontrolled approach was further developed in a spectacular way in a series of papers of Deng, Nahmod and Yue~\cite{deng-nahmod-yue-averaging,DNYHartree} and Sun, Tzvetkov~\cite{SunTzvetkov2021derivative} bringing tools from random matrix theory and introducing powerful methods such as the random averaging operators and the random tensors. More recently, the paracontrolled approach has taken a step forward, and were successfully implied in the resolution by Bringmann, Deng, Nahmod and Yue~\cite{bdny} of the $\phi_4^3$ problem for NLW.  
~\\

\paragraph{\bf Consequences of Theorem~\ref{theo:main} and perspectives} 

\begin{enumerate}[-]
    \item From the perspective of probabilistic Cauchy theory,
 equation~\eqref{eq:NLSHW} is the first dispersionless Schrödinger-type equation for which we prove probabilistic well-posedness in a quasilinear regime, where the second Picard's iteration does not gain any regularity. Indeed, we provide a measure induced by a \emph{mild} unit-scale and one-directional randomization procedure in part~\ref{eq:randomization}. In contrast with~\cite{schroheis3} on the Schrödinger equation on the Heisenberg group, the randomization does not gain any regularity in $L^p$-spaces and does not penalize the dispersionless direction. We believe that the present approach would also be successful in this context.
\item From the deterministic perspective, Theorem~\ref{theo:main} yields a statistical ensemble $\Sigma$ (a dense full-measure set) of initial data leading to strong solutions, in regimes where we show that the equation is semi-linearly ill-posed.  
The theorem also provides strong solutions in the energy space, which is a first modest step in the comprehension of the long-time Cauchy theory of~\eqref{eq:NLSHW}, motivated by the rich possible asymptotic behaviors of solutions studied in~\cite{Xu2017unbounded} in the defocusing case and~\cite{ BahriIbrahimKikuchi2021transverse, BahriIbrahimKikuchi2020remarks} in the focusing case. In contrast to the half-wave equation or the Szeg\H{o} equation on the line, local existence in the energy space $\mathcal{H}^\frac{1}{2}$ does not follow from a Yudovich argument since the $L^q$-norms are not controlled by the energy when $q>6$. We cannot use the conservation of the energy together with a Brezis-Gallouët estimate either in oder to show that in the defocusing case, smooth solutions extend globally in time, since $\mathcal{H}^s$ is not an algebra when $s<\frac{3}{4}$, and there is no conservation law that controls the $\mathcal{H}^\frac{3}{4}$-norm.
\item In addition, equation~\eqref{eq:NLSHW} can serve as a model to study the long-time behavior of probabilistic solutions generated by the paracontrolled decomposition, in the absence of an invariant measure. Indeed, we construct a probabilistic solution in the presence of a conserved energy which is not enough to globalize the solutions since \emph{probabilistic} information is crucial in the iteration scheme. However, to understand how this information is transported by the flow and to prove quasi-invariance of the measure, one needs dispersion (~\cite{OhSosoeTzvetkov}). Hence, understanding the long-time behavior of the solutions provided by Theorem~\ref{theo:main} requires the combination of energy methods with some measure-theory argument. This is a challenging but interesting problem. 
\end{enumerate}
~\\
\paragraph{\textbf{Strategy of the proof}}
The proof amounts to show the convergence of $(u_n)_{n\in\N}$ (whose existence is ensured by Theorem 1.6 in~\cite{BahriIbrahimKikuchi2020remarks} recalled in Proposition~\ref{prop:lwp}), on a fixed time interval. First, we establish the convergence along the subsequence $(u_N)_{N\in2^\N}$ by adapting the iterative scheme from~\cite{bringmann2021-ansatz}, which is made up of an induction on frequencies. At every step $N$ of the induction, we consider separately the problematic {\it high-low-low} frequency interactions, absorbed in the \emph{adapted linear evolution} $F_N$ studied in Section~\ref{sec:linear_evolution} and solution to
\[
\begin{cases}
i\partial_tF_N + (\partial_{xx}^2-|D_y|)F_N = \mathcal{N}(F_N,P_{\leq N^\gamma} u_{\frac{N}{2}},P_{\leq N^\gamma}u_{\frac{N}{2}})\,.\\
F_N(0) = P_Nf_0^\omega\,,
\end{cases}
\]
We get {\it probabilistic Strichartz estimates} for small times, controlling the $L^\infty$ norm of $F_N$ with a loss $N^{\frac\gamma2-\sigma}$ instead of $N^{\frac 12}$. To achieve such a goal, we exploit the probabilistic structure of the initial data $f_0^\omega$, which the sum of terms frequency localized in unit-scale intervals and decoupled by independent Gaussian variables. A key ingredient is the independence between the approximate solution $u_{\frac N2}$, which is constructed from the low modes  $P_{\leq \frac{N}{2}}f_0^\omega$ of the initial data, and $P_Nf_0^\omega$. As for the remainder $w_N$, it solves the equation with zero initial data and stochastic forcing terms but without the singular interaction
\[
\mathcal{N}_s(F_N,u_{\frac N2},u_{\frac N2}) :=\mathcal{N}(F_N,P_{\leq N^\gamma} u_{\frac{N}{2}},P_{\leq N^\gamma}u_{\frac{N}{2}})\,.
\]
In order to get convergence on a fixed time interval, not depending on the iteration step $N$, we follow~\cite{bringmann2021-ansatz} by making use of the {\it truncation method} from De Bouard and Debussche~\cite{DeBouardDebussche99} sketched in Section~\ref{sec:truncation}. We stress out that we only do this in the half-wave variable $y$, whereas a deterministic analysis is performed in the Schrödinger variable $x$. Namely, we exploit the dispersion materialized by the Strichartz estimates and use mixed Lebesgue spaces and $TT^*$-type argument instead of Gronwall inequalities and energy estimates obtained in~\cite{bringmann2021-ansatz}. 
~\\
\paragraph{\textbf{Outline of the paper}}

The paper is organized as follows. We first introduce the unit-scale one-directional randomization of the initial in Section~\ref{sec:prelim}, and prove refined Strichartz estimates for frequency localized functions. Then, we define in Section~\ref{sec:scheme} the dyadic subsequence of smooth approximate solutions $u_N$ and set up the truncation method. The adapted linear evolution is handled in Section~\ref{sec:linear_evolution}, and the a priori bounds for the nonlinear remainders are established in Section~\ref{sec:nonlinear_evolution}. These bounds rely on paracontrolled trilinear estimates that are detailed in Section~\ref{section:trilinear_proof}. Since the nonlinearity is cubic, we have to control more various types of interactions than in~\cite{bringmann2021-ansatz} where the nonlinearity is quadratic. 
Finally, in Section~\ref{sec:conclusion}, we show convergence of the approximating sequence $(u_N)_N$ to a limit $u$ on a fixed interval, and we prove that it solves the equation~\eqref{eq:NLSHW} on a smaller time interval, which is almost-surely nonempty. Then we deduce the convergence of the whole sequence $(u_n)_n$. We conclude by detailing some ill-posedness results for equation~\eqref{eq:NLSHW} in Appendix~\ref{sec:ill-posedness}.
~\\
\paragraph{\textbf{Notation}}
We denote the linear operator $\mathcal{A}=\partial_{xx}^2 - \abs{D_y}$, and the interactions stemming from the cubic nonlinear term by
\[
\mathcal{N}(u) = \abs{u}^2u\,,\quad \mathcal{N}(u_1,u_2,u_3) = \overline{u_1}u_2u_3 + u_1\overline{u_2}u_3 + u_1u_2\overline{u_3}\,.
\]

We write $\mathcal{F}$ or $\mathcal{F}_{y\to\eta}$ the partial Fourier transform with respect to the half-wave direction $y$, by $\eta=\widehat{y}$ the corresponding Fourier frequency, and by $\mathcal{F}^{-1}$ or $\mathcal{F}_{\eta\to y}^{-1}$ the inverse Fourier transform. For fixed $\alpha\in\R$, a parameter $\beta$ satisfies the relation $\beta=\alpha-0$ if there exists $\varepsilon>0$ such that $\beta=\alpha-\varepsilon$, where $\varepsilon>0$ can be chosen arbitrarily close to zero independently of the parameters. 

Given two Banach spaces $E,F$, we denote $\mathcal{B}(E\to F)$ the set of bounded operators from $E$ to $F$.
\section{Set up and preliminaries}\label{sec:prelim}

Since the standard Strichartz estimates for the one-dimensional free Schrödinger evolution provide control on the norm $L_x^\infty(\R)$ without losing derivatives, we only need to consider randomized initial data in the $y$-direction as in Figure~\ref{fig:Pnk}. We fix $\varphi\in\mathcal{C}_c^{\infty}(\R,[0,1])$ supported in $[-1,1]$ and equal to $1$ on $[-\frac 12,\frac 12]$ such that $\{\varphi(\cdot-k)\colon k\in\Z\}$ form a partition of unity, and set $\psi(x)=\varphi(x)-\varphi\left(\frac{x}{2}\right)$. For dyadic $M$, the spectral projector $P_{M,k}$ around $k\in\Z$ of size $M$ for $y$-frequencies is defined  as follows (see Figure~\ref{fig:Pnk}). For $f_0\in L^2(\R^2)$ we set
\begin{align*}
    P_{1,k}f_0&=\mathcal{F}_{y\to\eta}^{-1}\parent{\varphi(\eta-k)\mathcal{F}_{y\to\eta} f_0}\,,\\
P_{M,k}f_0&=\mathcal{F}_{y\to\eta}^{-1}\parent{\psi\left(M^{-1}(\eta-k)\right)\mathcal{F}_{y\to\eta}f_0}\,.
\end{align*}
Given a dyadic integer $N\geq1$, $P_N :=P_{N,0}$ denotes the Littlewood-Paley projector at frequencies $\frac{N}{2}\leq\eta\leq 2N$ in the $y$-direction only, and $\widetilde {P_N}$ denotes a similar fattened projector using a function~$\widetilde\psi$ with a slightly larger support. Let $(g_k)_k$ be a sequence of complex-valued independent normalized Gaussian variables.
Given $f_0\in \mathcal{H}^s(\R^2)$, we associate a random function obtained by the Wiener randomization
\begin{equation}
    \label{eq:randomization}
    f_0^\omega = \sum_{k\in\Z} g_{k}(\omega)P_{1,k}f_0\,.
\end{equation}
This randomization induces a non-degenerate probability measure on $\mathcal{H}^s(\R^2)$, densely supported (see~\cite{burq-tzvetkov-2008I} Lemma B.1). In addition, for and $N\in2^\N$, we set
\[
P_1f_0^\omega = g_0(\omega)P_1f_0\,,
\quad P_Nf_0^\omega = \sum_{\frac{N}{2}\leq\abs{k}< N} g_{k}(\omega)P_{1,k}f_0\,,
\quad P_{\leq N}f^\omega_0 = \sum_{M\leq N}P_{M}f^\omega_0\,.
\]
Finally, we emphasize that the family $\set{P_Nf_0^\omega}_{N\geq1}$ is jointly independent.
\begin{figurehere}
\label{fig:Pnk}
\begin{center}
\begin{tikzpicture}[scale=0.6]
\path [fill=GreenYellow] (-10,2) rectangle (10,8);
\path [fill=white] (-10,5.5) rectangle (10,4.5);
\draw[->, thick] (-10,0)--(10,0) node[right]{$\xi$};
\draw[->, thick] (0,-0.5)--(0,9) node[above]{$\eta$};
\draw (0,5) node{$+$} node[right]{$k$};
\draw[<->, thick, OliveGreen] (3,4.5)--(3,5.5);
\draw (3,5) node[right]{{\color{OliveGreen}$\frac{M}{2}$}};
\draw[<->, thick, OliveGreen] (5,2)--(5,8);
\draw (5,5) node[right]{{\color{OliveGreen}$2M$}};
\end{tikzpicture}
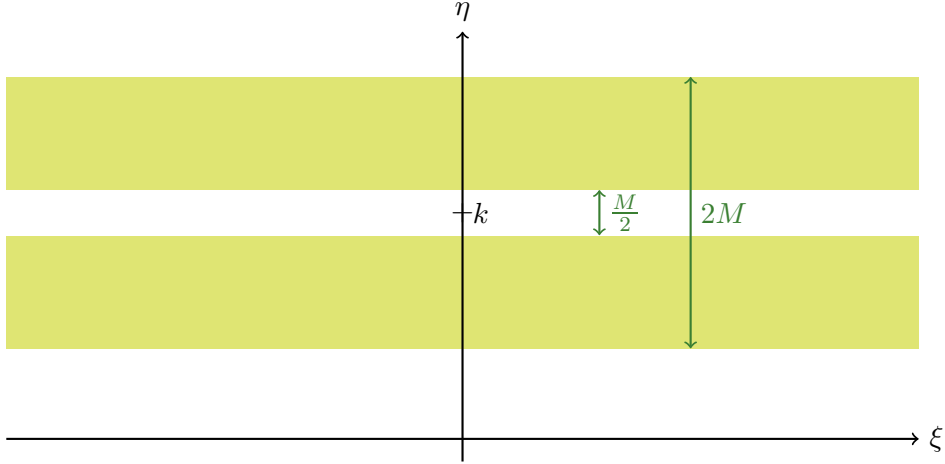
\caption{Support of the projection $P_{M,k}$ in the frequency space $(\R_\xi,\R_\eta)$.}
\end{center}
\end{figurehere}

\subsection{Dispersive estimates with a derivative loss}

We say that the pair $(p,q)$, with $2\leq p,q\leq\infty$, is admissible if it is a Schrödinger-admissible pair in dimension $1$, i.e.
\begin{equation}
\label{eq:adm}
    \frac{2}{p}+\frac{1}{q} = \frac{1}{2}\,.
\end{equation}
\begin{lem}[Bernstein - Strichartz estimates]
\label{lem:str} Given a Borel function $m:E\to \C$ with support in a bounded measurable set $E\subset\R$ of finite Lebesgue measure $\abs{E}<\infty$, we define $T_E$ to be the linear evolution  frequency-localized in $E$:
\[
T_E = \e^{it\A}\mathcal{F}_{\eta\to y}^{-1}(m(\eta) \mathcal{F}_{y\to\eta})\in\mathcal{B}(L_{x,y}^2(\R^2)\to L_{x,y}^2(\R^2) )\,,\quad t\in\R\,.
\]
Given $(p,q)$ and $(\widetilde{p},\widetilde{q})$ two admissible pairs as in~\eqref{eq:adm}, and $r\in[2,\infty]$, we have the following Strichartz estimates with a derivative loss:
\begin{itemize}
    \item Bernstein-Strichartz estimate
    \begin{equation}
        \label{eq:str}
        \norm{T_Ef_0}_{L_t^pL_x^qL_y^r} \lesssim \abs{E}^{\frac{1}{2}-\frac{1}{r}} \norm{f_0}_{L_{x,y}^2}\,,
    \end{equation}
    \item $T^*$-estimate
    \begin{equation*}
        \norm{\int_\R\e^{-i\tau\A}\mathcal{F}_{\eta\to y}^{-1}(m(\eta) \mathcal{F}_{y\to\eta})f_0(\tau)\dd\tau}_{L_{x,y}^2} \lesssim \abs{E}^{\frac{1}{2}-\frac{1}{r}} \norm{f_0}_{L_t^{p'}L_x^{q'}L_y^{r'}}\,,
    \end{equation*}
    \item $TT^*$-estimate
        \begin{equation}
        \label{eq:TTstar}
        \norm{\int_\R\mathbf{1}_{0\leq\tau\leq t} \e^{i(t-\tau)\A}\mathcal{F}_{\eta\to y}^{-1}(m(\eta) \mathcal{F}_{y\to\eta})f_0(\tau)\dd\tau}_{L_t^pL_x^qL_y^r} \lesssim \abs{E}^{\frac{1}{2}-\frac{1}{r}} \norm{f_0}_{L_t^{\widetilde{p}'}L_x^{\widetilde{q}'}L_y^2}\,.
    \end{equation}
\end{itemize}
\end{lem}
Notice that when $r=2$ there is no loss in the Bernstein estimate and we can choose $E=\R$.
\begin{proof}
In the mixed Lebesgue norms, we first use the unit-scale Bernstein estimate in the $y$-direction to pass from $L_y^r$ to $L_y^2$, and then the unitarity of $\e^{it\abs{D_y}}$ on $L_y^2$ to reduce the matter to the free Schrödinger evolution on the line. There holds
    \[
    \norm{T_Ef_0}_{L_t^pL_x^qL_y^r}\lesssim \abs{E}^{\frac{1}{2}-\frac{1}{r}}\norm{\e^{it\A}f_0}_{L_t^pL_x^qL_y^2}= \abs{E}^{\frac{1}{2}-\frac{1}{r}}\norm{\norm{\e^{it\partial_{xx}^2}f_0}_{L_y^2}}_{L_t^pL_x^q}\,.
    \]
Subsequently, we use the Minkowski inequality and the Strichartz estimate for $\e^{it\partial_{xx}^2}$ in the $x$-direction to conclude that
\[
\abs{E}^{\frac{1}{2}-\frac{1}{r}}\norm{\norm{\e^{it\partial_{xx}^2}f_0}_{L_y^2}}_{L_t^pL_x^q}\lesssim \abs{E}^{\frac{1}{2}-\frac{1}{r}}\norm{\norm{\e^{it\partial_{xx}^2}f_0}_{L_t^pL_x^q}}_{L_y^2}\lesssim \abs{E}^{\frac{1}{2}-\frac{1}{r}}\norm{f_0}_{L_{x,y}^2}\,.
\]
The $T^*$-estimate follows from duality, and the $TT^*$-estimate is a consequence of the first two estimates, by using the $T^*$ estimate with $r=2$. We also need an application of the standard Christ-Kiselev argument to localize in time  $0\leq\tau\leq t$.
\end{proof}

\subsection{Local existence above the energy space}

We recall the best know local well-posedness result for equation~\eqref{eq:NLSHW}.
\begin{prop}[Local well-posedness above the energy space,~\cite{BahriIbrahimKikuchi2020remarks} Theorem 1.6]
\label{prop:lwp}
Let $s>\frac 12$. For every $f_0\in \mathcal{H}^s(\R^2)$, there exists $T=T(\|f_0\|_{\mathcal{H}^s})>0$ such that equation~\eqref{eq:NLSHW} admits a unique local solution in $\mathcal{C}((-T,T),\mathcal{H}^s)$. 
\end{prop}

It is shown in~\cite{BahriIbrahimKikuchi2020remarks} that local well-posedness actually holds in $L^2_xH^s_y$. The proof follows from a fixed point argument in $L^\infty_t(\intervalcc{-T}{T};\mathcal{H}^s(\R^2))\cap L_t^4(\intervalcc{-T}{T};L_{x,y}^\infty(\R^2))$, for some $T>0$ depending on the $\mathcal{H}^s(\R^2)$-norm of the initial data. The use of the Strichartz space $L_t^4(\intervalcc{-T}{T};L_{x,y}^\infty(\R^2))$ requires that $s>\frac{1}{2}$ because of the lack of dispersion in the $y$-direction. Consequently, we have no control on the growth of the $\mathcal{H}^s$-norm of the solution since we do not have access to the conserved energy, which is at the level of $\mathcal{H}^\frac{1}{2}$. As discussed in the introduction, local existence for smooth solutions to~\eqref{eq:NLSHW} is an open problem. 

\section{Iteration scheme}\label{sec:scheme}

\subsection{The refined Ansatz}
We construct the solution to~\eqref{eq:NLSHW} by iteration on the frequencies, that we regroup as dyadic packets. Let $N\in2^\N$, and let $u_N$ be the maximal solution\footnote{Observe that the local well-posedness result from~\cite{BahriIbrahimKikuchi2020remarks}, recalled in Proposition~\ref{prop:lwp}, implies existence and uniqueness for $u_N$ in $\mathcal{H}^{\frac{1}{2}+}(\R^2)$ up to a time $T(n)>0$. In Theorem~\ref{theo:main} we claim that $(u_N)$ converges almost-surely in $\mathcal{H}^s(\R^2)$ to a strong solution to~\eqref{eq:NLSHW} up to a time $T^\omega>0$ almost-surely in the randomization.} to~\eqref{eq:NLSHW} with initial data localized at frequencies at most $N$
\begin{equation*}
    \begin{cases}
    (i\partial_t + \A)u_N = \abs{u_N}^2u_N\,,\\
    u_N(0) = P_{\leq N}f_0^\omega\,.
    \end{cases}
\end{equation*}
We write $u_N = u_{\frac N2}+v_N$, so that $v_N(0)$ is localized at frequency $\sim N$, and $v_N$ solves
\begin{equation*}
    \begin{cases}
    (i\partial_t + \A)v_N = \mathcal{N}\parent{v_N+u_{\frac N2}} - \mathcal{N}\parent{u_{\frac N2}}\,,\\
    v_N(0) = P_Nf_0^\omega\,.
    \end{cases}
\end{equation*}
Then, we consider at each step $n$ the adapted \textit{linear-nonlinear decomposition} $v_N = F_N + w_N$, defined as follows. 
\begin{itemize}
    \item The adapted linear evolution $F_N$ is solution to a linear Schrödinger half-wave equation that encapsulates the bad nonlinear interactions, with initial data $P_Nf_0^{\omega}$ localized at frequency $\sim N$:
\begin{equation*}
    \begin{cases}
    (i\partial_t+\A)F_N = \mathcal{N}\parent{F_N, P_{\leq N^\gamma}u_{\frac N2}, P_{\leq N^\gamma} u_{\frac N2}}\,,\\
    F_N(0) = P_Nf_0^\omega\,.
    \end{cases}
\end{equation*}
The parameter $0<\gamma<1$, to be determined later in the analysis, separates two ranges of frequency scales $N\gg N^{\gamma}$.
\item The nonlinear remainder~$w_N$ is solution to~\eqref{eq:NLSHW} with a stochastic forcing term and zero initial condition:
\begin{equation}
    \label{eq:wn}
    \begin{cases}
    (i\partial_t+\A)w_N = \mathcal{N}\parent{u_N} - \mathcal{N}\parent{u_{\frac N2}} - \mathcal{N}\parent{F_N,P_{\leq N^\gamma}u_{\frac N2},P_{\leq N^\gamma}u_{\frac N2}} \,,\\
    w_N(0) = 0\,.
    \end{cases}
\end{equation}
\end{itemize}
The key point is that the singular interactions between the \emph{high} frequencies of the initial data and the \emph{low} frequencies of the approximate solution $u_{N/2}$ are removed from the equation satisfied by $w_N$. Therefore, we can expect that a smoothing effect now occurs on $w_N$. Note that at each step we can decompose the solution as a series
\begin{equation}\label{eq:un_series}
u_N = u_{\frac N2}+ F_N+w_N = u_{N_0} + \sum_{L=2N_0+1}^N \left(w_L + F_L^\omega\right)\,.
\end{equation}
The goal is to prove that the series with general term $(w_N)_{N\geq N_0}$ almost-surely converges in a \textit{subcritical} space $\mathcal{C}(\intervalcc{-T_0}{T_0}; \mathcal{H}^\nu(\R^2))$, for some $\nu>\frac{1}{2}$ and some $0<T_0\ll 1$, and the series with general term $(F_N)_{N\geq N_0}$ converges almost-surely in $\mathcal{C}(\intervalcc{-T_0}{T_0}; \mathcal{H}^s(\R^2))$. In addition, we prove that there exists a time $T^\omega>0$ such that the limit of $(u_N)$ solves~\eqref{eq:NLSHW} in $\mathcal{C}\parent{\intervalcc{-T^\omega}{T^\omega};\mathcal{H}^s(\R^2)}$.

\subsection{Functional spaces}

In the nonlinear analysis, we need to keep track of the frequency localization of the adapted linear evolution $F_N$. Specifically, in order to avoid the \emph{high-low-low} interactions in the nonlinear analysis that are responsible for the lack of nonlinear smoothing, we want to ensure that for a given $N$, the $y$-frequencies of the functions $w_N$ and $F_N$ are localized at size $\abs{\eta}\sim N$. To capture such a frequency localization, rigorously stated in Lemmas~\ref{lem:loc-Fnk} and~\ref{lem:loc-Fn}, we fix a parameter $D>0$ and some weights 
\[
C_{N,D}(M) := \max\parentbig{\frac{N}{M},\frac{M}{N}}^D\,,\quad C_{\leq N,D}(M) := \max\parentbig{1,\frac{M}{N}}^D\,.
\]
Then we set the weighted norms
\begin{equation}
    \label{eq:cND}
\norm{u}_{X_{N,D}\intervalcc{-T}{T}} := \sum_{M\geq1}C_{N,D}(M)
	\norm{P_Mu}_{L_t^\infty(\intervalcc{-T}{T},L_{x,y}^2(\R^2))}
	\,,
\end{equation}
\[
\norm{u}_{X_{\leq N,D}\intervalcc{-T}{T}}
 := \sum_{M\geq1}C_{\leq N,D}(M)
 \norm{P_Mu}_{L_t^\infty(\intervalcc{-T}{T},L_{x,y}^2(\R^2))}
 \,.
\]
For some $r\in\intervalcc{2}{\infty}$, we define as well some frequency-localized mixed Strichartz norms 
\begin{equation*}
\begin{split}
\norm{u}_{S_{N,D}^r\intervalcc{-T}{T}} &:= \sum_{M\geq1}C_{N,D}(M)\norm{P_Mu}_{L_t^8(\intervalcc{-T}{T},L_x^4L_y^r(\R^2))}\,,\\
\norm{u}_{S_{\leq N,D}^r\intervalcc{-T}{T}} &:= \sum_{M\geq1}C_{\leq N,D}(M)\norm{P_Mu}_{L_t^8(\intervalcc{-T}{T},L_x^4L_y^r(\R^2))}\,.
\end{split}
\end{equation*}
When $r=\infty$ we simply write
\[
S_{N,D} := S_{N,D}^\infty\,,\quad S_{\leq N, D}:= S_{\leq N, D}^\infty\,.
\]
The frequency localization of $F_N^{\omega}$ around frequencies $N$ stems from the structure of $F_N$. Specifically, each $F_N$ is a superposition of functions $(F_{N,k})_{\frac{N}{2}\leq \abs{k}\leq N}$, decoupled by independent Gaussian variables. We observe that $F_{N,k}$ is solution to a linear Schrödinger equation with a potential truncated at frequencies of size $\lesssim N^\gamma$, with initial data $P_{1,k}f_0$ for some $\abs{k}\sim N$. Hence, at least in short-time, we expect $F_{N,k}$ to be localized at a distance $\lesssim N^\gamma$ from the frequency $k$. In particular,
\begin{equation}
\label{eq:loc-Fnk}
    \abs{\supp \mathcal{F}(F_{N,k})}\lesssim N^{\gamma}\quad \text{a.s.}
\end{equation}
This key frequency localization would reduce the derivative loss in the Strichartz estimates from Lemma~\ref{lem:str} for $F_{N,k}$, and therefore for $F_N$ by using probabilistic decoupling. To encapsulate~\eqref{eq:loc-Fnk}, rigorously stated in Lemma~\ref{lem:loc-Fnk}, we follow~\cite{bringmann2021-ansatz} and we consider the recentered Besov spaces
\begin{equation}
    \label{eq:norm:B}
    \norm{f}_{B_{k,D}^{\rho,\gamma}\intervalcc{-T}{T}} := \sum_{M\geq1}c_{k,D}^{\rho,\gamma}(M)\norm{P_{M,k}f}_{L_t^\infty(\intervalcc{-T}{T};L_{x,y}^2(\R^2))}\,,\quad c_{k,D}^{\rho,\gamma}(M) := M^\rho \max\parentbig{1,\frac{M}{N^\gamma}}^D\,.
\end{equation}
The weight $M^\rho$ accounts for an extra gain of derivative $\rho = \sigma-0$ on $F_{N,k}$ that comes from the frequency localization of $P_{M,k}F_{N,k}$. In Section~\ref{sec:conclusion}, we implement a fixed point argument on the nonlinear term $w$ in the space associated with the norm
\[
\norm{w}_{Y_N^\nu} 	= \norm{\japbrak{D_y}^\nu w}_{X_{N,\alpha}} + \norm{\japbrak{D_y}^\nu w}_{X_{\leq N,D}} 
	+ \norm{\japbrak{D_y}^\sigma w}_{S_{N,\alpha}} + \norm{\japbrak{D_y}^\sigma w}_{S_{\leq N,D}}\,,
\]
where $\alpha, D$ are parameters, $\nu>\frac{1}{2}$ is the regularity of the remainder $w_N$ determined by $0<\sigma=\nu-\frac{1}{2}-0$ from the Sobolev embedding $H^\nu_y\hookrightarrow W^{\sigma, \infty}_y$.  

\subsection{The recurrence and the truncation method from de Bouard and Debussche}
\label{sec:truncation}

In contrast with the standard probabilistic well-posedness argument implemented in~\cite{burq-tzvetkov-2008I}, which consists in a single fixed point argument for a fixed  $\omega\in\Omega$, we cannot solve the equation pathwise. The reason is that we need to perform a fixed point argument at each step of the iteration scheme, leading to a sequence of times $(T_N(\omega))_N$ that would depend intricately on $\omega$. In particular, we cannot ensure that this sequence is uniformly bounded from below in $n$.  Instead, we do have to work in a space $L^\rho(\Omega,\mathcal{E})$ for some $\rho\geq1$, where $\mathcal{E}$ is the space-time functional space in which we establish local existence. 

To handle a similar problem arising in the study of a Schrödinger equation with a multiplicative white noise~\cite{DeBouardDebussche99}, de Bouard and Debussche truncated the nonlinearity. This allows to perform a fixed-point argument in some space $L^\rho(\Omega,\mathcal{E})$ despite the absence of gain of integrability in the $\omega$-variable. We refer to~\cite{bringmann2021-ansatz} for some discussion on motivation for the truncation method of de Bouard and Debussche. The method consists in fixing a smooth function  $\theta\in\mathcal{C}_c^{\infty}(\R,[-1,1])$ compactly supported in $(-2,2)$ and such that $\theta\equiv 1$ on $[-1,1]$, and to essentially reduce the matter to a linear problem by assuming most of the terms to be at most $1$ in the nonlinearity. Let $F_{N,\theta}$ and $w_{M,\theta}$ to be determined later as the solutions to some equation. The cutoff functions are denoted
\begin{equation}
    \begin{split}
    \label{eq:theta_Fw}
\theta_{F;N}(\tau)&=\theta\left(\|\langle D_y\rangle^{\sigma'}F_{N,\theta}\|_{S_{\leq N,D'}[-\tau,\tau]}\right)\,,
\\
\theta_{w; N}(\tau)&=\theta\left(\|\langle D_y\rangle^{\sigma}w_{N,\theta}\|_{Y_N^{\nu}[-\tau,\tau]}\right)\,,\\
\theta_{F,w;\leq \frac N2}(\tau)&=\theta\parentbig{\sum_{M}^{N/2}\|\langle D_y\rangle^{\sigma'}F_{M,\theta}^{\omega}\|_{S_{\leq M,D'}[-\tau,\tau]}
	+\|\langle D_y\rangle^{\sigma}w_{M,\theta}\|_{S_{\leq M,D}[-\tau,\tau]}
	\\
	&\hspace{200pt}+\|\langle D_y\rangle^{\nu}w_{M,\theta}\|_{X_{\leq M,D}[-\tau,\tau]}}\,.
\end{split}
\end{equation}
We define by induction  $F_{N,\theta}$ and $w_{N,\theta}$ as solutions to the following truncated  equations.
First,
\begin{multline}
    \label{eq:trun-Fn}
    F_{N,\theta} (t) = \e^{it\A}P_Nf_0^\omega 
    + i\int_0^t\e^{i(t-\tau)\A}\theta_{F,w;\leq \frac N2}(\tau)^2 \mathcal{N}\parent{F_{N,\theta},P_{\leq N^\gamma}u_{\frac N2,\theta},P_{\leq N^\gamma}u_{\frac N2,\theta}}(\tau)\dd\tau\,,  
\end{multline}
and similarly, the block$F_{N,k,\theta}$ is solution to the decomposed truncated equation
\begin{multline}
    \label{eq:trun-Fnk}
    F_{N,k,\theta} (t) = \e^{it\A}P_{1,k}f_0 
    + i\int_0^t\e^{i(t-\tau)\A}\theta_{F,w;\leq \frac N2}(\tau)^2 \mathcal{N}\parent{F_{N,k,\theta},P_{\leq N^\gamma}u_{\frac N2,\theta},P_{\leq N^\gamma}u_{\frac N2,\theta}}(\tau)\dd\tau\,, 
\end{multline}
so that 
\begin{equation}\label{eq:F_n,theta_decompo}
F_{N,\theta} = \sum_{\frac{N}{2}\leq \abs{k}<N} g_k(\omega)F_{N,k,\theta}\,.
\end{equation}
Then, the truncated nonlinear remainder $w_{N,\theta}$  solves
\begin{equation}\label{eq:wtheta}
w_{N,\theta}(t)=i\int_{0}^t e^{i(t-\tau)\mathcal{A}}\widetilde{\mathcal{N}}_{\theta}(w_{N,\theta})(\tau)\dd \tau\,,
\end{equation}
where the nonlinearity $\widetilde{\mathcal{N}_{\theta}}(w_{N,\theta})$ is given by
\begin{equation}\label{eq:Ntilde}
    \begin{split}
    \widetilde{\mathcal{N}}_{\theta}(w_{N,\theta})
	&=\theta_{F;N}^2\mathcal{N}(F_{N,\theta})
	+\theta_{F;N}^2\mathcal{N}(F_{N,\theta},F_{N,\theta},w_{N,\theta})
	+\theta_{F;N}\theta_{w;N}\mathcal{N}(F_{N,\theta},w_{N,\theta},w_{N,\theta})\\
	&+\theta_{w;N}^2\mathcal{N}(w_{N,\theta})
	+\theta_{F,w;\leq \frac N2}^2 \mathcal{N}(u_{\frac N2,\theta},u_{\frac N2,\theta},w_{N,\theta})
	+\theta_{F,w;\leq \frac N2}\theta_{w;N} \mathcal{N}(u_{\frac N2,\theta},w_{N,\theta},w_{N,\theta})\\
	&+2\theta_{F,w;\leq \frac N2}\theta_{F;N} \mathcal{N}(u_{\frac N2,\theta},w_{N,\theta},F_{N,\theta})
	+\theta_{F,w;\leq \frac N2}\theta_{F;N}\mathcal{N}(F_{N,\theta},F_{N,\theta},u_{\frac N2,\theta})\\
	&+\theta_{F,w;\leq \frac N2}^2\left(\mathcal{N}(F_{N,\theta},u_{\frac N2,\theta},u_{\frac N2,\theta})-\mathcal{N}(F_{N,\theta},P_{\leq N^{\gamma}}u_{\frac N2,\theta},P_{\leq N^{\gamma}}u_{\frac N2,\theta})\right)\,.
    \end{split}
\end{equation}
Finally, we denote the truncated approximate solution
\[
u_{N,\theta}=u_{\frac N2,\theta}+F_{N,\theta}+w_{N,\theta}\,.
\]
We prove the convergence of $(u_{N,\theta})_N$ on a time-interval $\intervalcc{-T_0}{T_0}$. Subsequently, we use a priori estimates and a continuity argument to remove the truncation, and to get a solution to the original equation~\eqref{eq:wn} on a sub-interval $\intervalcc{-T^\omega}{T^\omega}$ with $0<T^\omega\leq T_0$ almost-surely. This is the matter of Section~\ref{sec:conclusion}, once we solved the fixed point for the truncated equations and proved the convergence of the approximate solutions $(u_{N,\theta})_N$ in $L^2(\Omega; \mathcal{E})$.

\section{The adapted linear evolution}\label{sec:linear_evolution}

From now on, and until Section~\ref{sec:conclusion}, we often drop the index $\theta$ for simplicity in the notation. At each step $n$, the adapted linear evolution is decomposed as~\eqref{eq:F_n,theta_decompo}
\[
F_N = \sum_{\frac{N}{2}\leq\abs{k}<N} g_{k}(\omega) F_{N,k}\,,
\]
where each unit block $F_{N,k}$ is the solution to the linear equation
\begin{equation*}
    \begin{cases}
    (i\partial_t+\A)F_{N,k} = \theta_{F,w;\leq \frac N2}^2\mathcal{N}(F_{N,k},P_{\leq N^\gamma}u_{\frac N2},P_{\leq N^\gamma}u_{\frac N2})\,,\\
    F_{N,k}(0) = P_{1,k}f_0\,.
    \end{cases}
\end{equation*}
We denote the \emph{potential} by $\phi:=\theta_{F,w;\leq \frac N2}P_{\leq N^\gamma}u_{\frac N2}$, and we use the following features of $\phi$ in this section:
\begin{itemize}
    \item $\phi$ is localized at low frequencies $\lesssim N^\gamma$;
    \item $\phi$ satisfies the following a priori estimate, encoded in the truncation defined in~\eqref{eq:theta_Fw}
    \[
    \sum_M\norm{\japbrak{D_y}^\sigma\phi}_{S_{M,D}\intervalcc{-\tau}{\tau}} \leq 1\,,\quad \tau \in\R\,;
    \]
    \item $\phi$ is measurable in the $\sigma$-algebra $\mathcal{B}_{<\frac{N}{2}}$ generated by the random variables $(g_k)_{\abs{k}<\frac{N}{2}}$. 
     Therefore, $\phi$ is independent of the Gaussian variables $(g_k)_{\frac{N}{2}\leq\abs{k}<N}$.
\end{itemize}
\subsection{Almost frequency localization}

 In what follows, we fix $n\geq1$, $N=2^n$, $k\in\Z$ with $\frac{N}{2}\leq\abs{k}<N$, and the potential $\phi =\theta_{F,w;\leq \frac N2}P_{\leq N^\gamma}u_{\frac N2}$ is as above. First, we prove the key frequency localization of $F_{N,k}$ at distance $\lesssim N^\gamma$ from $k$, with a gain of $\rho=\sigma-0$ derivatives captured by the recentered Besov spaces $B_{k,D}^{\rho,\gamma}$ defined in~\eqref{eq:norm:B}. This stems from the localization of the initial condition $P_{1,k}f_0$ in a unit-scale interval around $k$, together with the localization of the potential $\phi$.

\begin{lem}[Frequency localization of the individual blocks]
\label{lem:loc-Fnk}
Let $D''\gg1$. Then, for $0\leq \rho<\sigma$ arbitrarily close to $\sigma$, there exist $0<T_0$ and $0<C$ such that for every $N$ and $\frac N2\leq |k| < N$ the solution $F_{N,k}$ to~\eqref{eq:trun-Fnk} satisfies
\begin{equation}
\label{eq:loc-Fk}
    \norm{F_{N,k}}_{ B_{k,D''}^{\rho,\gamma}\intervalcc{-T_0}{T_0}}
    \leq C\norm{P_{1,k}f_0}_{L^2_{x,y}}\,.
\end{equation}
As a consequence,  
\begin{equation}
\label{eq:loc-Fk-sum}
    \norm{\mathbf{1}_{\frac N2\leq |k|<N} F_{N,k}}_{\ell_k^2(\Z ; B_{k,D''}^{\rho,\gamma}\intervalcc{-T_0}{T_0})}
    \leq C\norm{\widetilde{P_N}f_0}_{L^2_{x,y}}\,.
\end{equation}
\end{lem}
The proof follows the same lines as the proof of Proposition 4.1 in~\cite{bringmann2021-ansatz}. Nevertheless, we use the Duhamel formula and the Strichartz estimate in place of a Gronwall argument in order to avoid using the quantity $L_{x,y}^\infty(\R^2)$ that we do not control.
\begin{proof}
The estimate~\eqref{eq:loc-Fk-sum} follows from~\eqref{eq:loc-Fk} by using that 
\[
\norm{\mathbf{1}_{\frac N2\leq |k|<N}P_{1,k}f_0}_{\ell_k^2(\Z;L_{x,y}^2(\R^2))}\lesssim \norm{\widetilde{P}_Nf_0}_{L_{x,y}^2}\,.
\]
We prove~\eqref{eq:loc-Fk}: given $0<T_0\leq 1$, we show that for all $D''>0$, there exists $C>0$ such that for all $k\in\Z$ and $0\leq T\leq T_0$,
\begin{equation}        
\label{eq:loc-Fkbis}
\norm{ F_{N,k}}_{B_{k,D''}^{\rho,\gamma}\intervalcc{-T}{T}}
    	\leq \norm{ P_{1,k}f_0}_{L^2_{x,y}}+ C T^\frac{1}{2}\norm{ F_{N,k}}_{B_{k,D''}^{\rho,\gamma}\intervalcc{-T}{T}}\norm{ \japbrak{D_y}^\sigma\phi}_{L_t^8L_x^4L_y^\infty}^2\,.
\end{equation}
We then use the truncation $\theta$ to see that the norms of the potential is less than one, and we chose $T_0>0$ (independent of $N$) such that $CT_0^{\frac{1}{2}}\leq \frac{1}{2}$.
Let us establish~\eqref{eq:loc-Fkbis} for fixed $k\in\Z$. First, observe that 
\[
\norm{P_{1,k}f_0}_{B_{k,D''}^{\rho,\gamma}[-T,T]}\lesssim\|P_{1,k}f_0\|_{ L^2_{x,y}}\,.
\]
Next, it follows from the definition of the Besov norm~\eqref{eq:norm:B} and from the Duhamel formula that we need to estimate 
\[
\sum_M c_{k,D''}^{\rho,\gamma}(M)\norm{\int_{0}^{t} e^{i(t-\tau)\mathcal{A}}  P_{M,k}\parent{\mathcal{N}(F_{N,k},\phi,\phi)}(\tau)\dd\tau }_{L_{x,y}^2}\,.
\]
We choose the Strichartz admissible pairs $(p,q)=(\infty,2)$, $(\widetilde{p},\widetilde{q})=(4,\infty)$ and $r=2$, then we apply the $TT^*$-estimate from Lemma~\ref{lem:str} (recall that when $r=2$ there is no need to apply the Bernstein estimate, and there is no derivative loss). For fixed $M$, we have
\[
\norm{\int_{0}^{t} e^{i(t-\tau)\mathcal{A}}  P_{M,k}\parent{\mathcal{N}(F_{N,k},\phi,\phi)}(\tau)\dd\tau }_{L_t^\infty L_{x,y}^2}\lesssim\norm{P_{M,k}\parent{\mathcal{N}(F_{N,k},\phi,\phi)} }_{L_t^{\frac{4}{3}}L_x^1L_y^2}\,.
\]
We do a Littlewood-Paley decomposition of $F_{N,k}(t)$ with frequencies $K$ centered around $k$, and a standard Littlewood-Paley decomposition of $\phi$ with frequencies $L_1,L_2$. Without loss of generality, we assume that $L_2\leq L_1$. There holds
\begin{equation*}
   \norm{P_{M,k}\parent{\mathcal{N}(F_{N,k},\phi,\phi)}}_{L_t^\frac{4}{3}L_x^1L^{2}_y}
    \lesssim  \sum_K\sum_{L_2\leq L_1\lesssim N^\gamma} \norm{  P_{M,k}\parent{\mathcal{N}(P_{K,k}F_{N,k},P_{L_1}\phi,P_{L_2}\phi)}}_{L_t^\frac{4}{3}L_x^1L^{2}_y} \,.
\end{equation*}
Then, the interactions that contribute to the above sum are the following:
\begin{itemize}
    \item $K\sim M$,
    \item $K\ll M$ and $L_1\sim M$,
    \item $K\gg M$ and $L_1\sim K$,

\end{itemize}
Hence, 
\begin{multline*}
\norm{  P_{M,k} \parent{\mathcal{N}(F_{N,k},\phi,\phi)}}_{L_t^\frac{4}{3}L_x^1L_y^2}
    \lesssim T^\frac{1}{2}\Big(\sup_LL^\sigma\norm{P_L\phi}_{L_t^8L_x^4L_y^{\infty}}\Big)^2\\
    \Big(\sum_{K\sim M}\norm{  P_{K,k}F_{N,k}}_{L_t^\infty L_{x,y}^2}
    +\sum_{K:\, K\ll M\sim L_1}L_1^{-\sigma}\norm{ P_{K,k}F_{N,k}}_{L_t^\infty L^2_{x,y}}
    +\sum_{K:\, K\sim L_1\gg M}L_1^{-\sigma}\norm{  P_{K,k}F_{N,k}}_{L_t^\infty L^2_{x,y}}\Big)\,.
\end{multline*}
It follows from the definition~\eqref{eq:norm:B} of weights $c_{k,D''}^{\rho,\gamma}(K)$ and from the truncation~\eqref{eq:trun-Fnk} that
\[
\sup_LL^\sigma\norm{P_L\phi}_{L_t^8L_x^4L_y^{\infty}}\lesssim 1\,.
\]
Therefore, 
\begin{multline*}
\norm{ P_{M,k}\parent{\mathcal{N}(F_{N,k},\phi,\phi)}}_{L_t^\frac{4}{3}L_x^1L^2_y}
   \lesssim T^\frac{1}{2}\sum_{K\sim M}\norm{  P_{K,k}F_{N,k}(t)}_{L_{x,y}^2}\\
    +T^\frac{1}{2}\parentbig{\sum_{K\ll M\sim L_1} K^{-\rho}L_1^{-{\sigma}} 
    + \sum_{L_1\sim K\gg M} K^{-\rho}\max\left(1,\frac{K}{N^\gamma}\right)^{-D''}L_1^{-{\sigma}}}
    \norm{  F_{N,k}(t)}_{B_{k,D''}^{\rho,\gamma}\intervalcc{-T}{T}}
    \,.
\end{multline*}
Next, we multiply by $c_{k,D''}^{\rho,\gamma}(M)=M^{\rho}\max\left(1,\frac{M}{N^{\gamma}}\right)^D$ and we sum over $M$ to obtain
\begin{multline*}
\norm{ F_{N,k}(t)}_{B_{k,D''}^{\rho,\gamma}\intervalcc{-T}{T}}
 \lesssim \norm{P_{1,k}f_0}_{L^2_{x,y}} + T^\frac{1}{2}\sum_{M\sim N} c_{k,D''}^{\rho,\gamma}(M)\norm{P_{M,k}F_{N,k}}_{L_t^\infty L_{x,y}^2} \\
 +T^\frac{1}{2}\left(\sum_{M\lesssim N^\gamma} M^{\rho-\sigma}\right)\norm{F_{N,k}}_{B_{k,D''}^{\rho,\gamma}\intervalcc{-T}{T}}\,.
\end{multline*}
The assumption that $\sigma>\rho$ allows us to 
conclude.
\end{proof}
Now, we prove that $F_N$ is essentially localized at frequency $N$. This turns out to be crucial in the estimates for the remainder $w_N$, in the situations when we want to avoid the \emph{high-low-low} frequency interactions. We recall that such a frequency localization is captured by the weighted norm~\eqref{eq:cND}. 

\begin{lem}[Frequency localization for the adapted evolution]
\label{lem:loc-Fn}
Let $D''>0$. There exists $T_0>0$ such that $F_N$ satisfies
\[
\norm{ \japbrak{D_y}^{s} F_{N}^\omega}_{X_{N,D'}\intervalcc{-T_0}{T_0}}
	 \lesssim  \norm{  \widetilde{P_N}f_0^\omega}_{\mathcal{H}^s(\R^2)}\,.
\]
\end{lem}

\begin{proof}
Given $M\gg1$ and $0\leq T\leq 1$, we need to estimate 
\[
M^sC_{N,D'}(M)\norm{  P_MF_N}_{L^{\infty}_t([-T,T];L_{x,y}^2)}\,.
\]
We recall that $F_N$ is solution to the truncated equation~\eqref{eq:trun-Fn}, and that $\phi(t) =\theta_{F,w;\leq \frac N2} P_{\leq N^\gamma} u_{\frac N2 }(t)$. 
-\emph{Case when $M\sim N$}. In contrast with~\cite{bringmann2021-ansatz}, we do not use an energy estimate. Instead, we proceed as in the proof of Lemma~\ref{lem:loc-Fnk} in order to avoid using $L_{x,y}^\infty$. We start from the Duhamel formula, and we estimate
\begin{equation*}
    F_N(t) =  P_{N}\e^{it\A}f_0^{\omega} + \int_0^t \e^{i(t-\tau)\A} \mathcal{N}(F_N,\phi,\phi)(\tau)\dd\tau\,.
\end{equation*}
From the Bernstein-Strichartz $TT^*$-estimate in Lemma~\ref{lem:loc-Fnk}, with $(p,q)=(\infty,2)$, $(\widetilde{p},\widetilde{q})=(4,\infty)$ and $r=2$, we obtain
\begin{equation*}
 \|  F_N\|_{L_t^\infty L^2_{x,y}}
 	\lesssim \| P_N f_0^{\omega}\|_{L^2_{x,y}} +\|\mathcal{N}(F_N,\phi,\phi)\|_{L^\frac{4}{3}_tL^1_xL^2_y}\,.
\end{equation*}
Thanks to Hölder's inequality and to the truncation, we know that on $[-T,T]$,
\[
\|\mathcal{N}(F_N,\phi,\phi) \|_{L^\frac{4}{3}_tL^1_xL^2_y}
	\lesssim T^\frac{1}{2}\|F_N\|_{L^{\infty}_t L^2_{x,y}}\|\phi\|_{L^{8}_tL^4_xL^\infty_y}^2 \lesssim T^\frac{1}{2}\|F_N\|_{L^{\infty}_t L^2_{x,y}} \,.
\]
Choosing $T_0$ small enough we conclude that if $T\leq T_0$,
\[
\norm{F_N}_{L_t^\infty L_{x,y}^2} \lesssim \norm{P_Nf_0^\omega}_{L_{x,y}^2}\,.
\]
Therefore, when $M\sim N$, we deduce that
\[
M^{s}\norm{ P_MF_N}_{L_t^\infty L_{x,y}^2}\lesssim N^{s}\norm{ F_N}_{L_t^\infty L_{x,y}^2}\lesssim \norm{\japbrak{D_y}^{s} \widetilde{P_N}f_0^\omega}_{L_{x,y}^2}\,.
\]

-\emph{Case when $M\nsim N$}. In this case, we first fix $k$ with $\frac{N}{2}\leq \abs{k}< N$ and we perform a dyadic decomposition in frequency around $k$ for $F_{N,k}$. We see that when $M\ll N$,
\[
\norm{P_M F_{N,k}}_{L_t^\infty L_{x,y}^2} \lesssim \sum_{K\sim N} \norm{P_{K;k}F_{N,k}}_{L_t^\infty L_{x,y}^2} \lesssim \parentbig{\frac{N}{N^\gamma}}^{-D''} \norm{F_{N,k}}_{B_{k,D''}^{\rho,\gamma}}\,,
\]
whereas when $N\ll M$, we have 
\[
\norm{P_M F_{N,k}}_{L_t^\infty L_{x,y}^2} \lesssim \sum_{K\sim M} \norm{P_{K;k}F_{N,k}}_{L_t^\infty L_{x,y}^2} \lesssim \parentbig{\frac{M}{N^\gamma}}^{-D''} \norm{F_{N,k}}_{B_{k,D''}^{\rho,\gamma}}\,.
\]
Hence, 
\[
\norm{P_M F_{N,k}}_{L_t^\infty L_{x,y}^2} \lesssim \parentbig{\frac{\max(M,N)}{N^\gamma}}^{-D''}\norm{F_{N,k}}_{B_{k,D''}^{\rho,\gamma}}\,.
\]
We deduce that if $D''$ is large enough with respect to $D'$, then
\[
C_{N,D'}(M)\| P_MF_N(t)\|_{L_t^\infty L^2_{x,y}}
	\lesssim (MN)^{-D'} \sum_{\frac N2\leq k<N}|g_k(\omega)|\| F_{N,k}\|_{B^{\rho,\gamma}_{k,D''}}\,.
\]
Using Lemma~\ref{lem:loc-Fnk} we deduce that
\[
C_{N,D'}(M)\| P_MF_N(t)\|_{L_t^\infty L^2_{x,y}}
	\lesssim (MN)^{-D'} \sum_{\frac N2\leq k<N}|g_k(\omega)|\| P_{1,k}f_0\|_{L^2_{x,y}}\,.
\]
From the Cauchy-Schwarz inequality and using that $D'$ can be taken large enough to control the extra $N^\frac{1}{2}$-loss coming to the sum over $k$ and the factor $M^s$, we conclude
\[
M^s C_{N,D'}(M)\| P_MF_N(t)\|_{L^2_{x,y}}
	\lesssim (MN)^{-\frac{D'}{2}}\| \widetilde{P_N}f_0^{\omega}\|_{L^2_{x,y}} \,.\qedhere
\]
\end{proof}

\subsection{Strichartz estimates for the adapted linear evolution}

In this part, we undercut the derivative loss in the Strichartz estimates~\eqref{eq:str} for the adapted linear evolution $F_N$. To do so, we take advantage of the frequency localization of the function $F_{N,k}$ together with probabilistic decoupling.
\begin{lem}[Strichartz estimate for the individual block]\label{lem:strichartz_block}
Let  $0<\rho<\sigma$ as in Lemma~\ref{lem:loc-Fnk}, let $\delta=\gamma(\sigma-\rho)>0$ (arbitrarily small as $\rho$ is close to $\sigma$), and let $T_0>0$ small enough. For all $D'>0$, there exists $D''> D'$ such that for all $\sigma'>\sigma$ and $\frac{N}{2}\leq \abs{k}<N$, $0<T\leq T_0$, we have
\begin{equation*}
    \norm{\japbrak{D_y}^{\sigma'}F_{N,k}}_{S_{N,D'}^r\intervalcc{-T}{T}}
    \lesssim_r \norm{\japbrak{D_y}^{\sigma'}P_{1,k}f_0}_{L_{x,y}^2}
    + T^{\frac{1}{2}}N^{\sigma'+\gamma(\frac{1}{2}-\frac{1}{r})-\gamma\sigma+\delta}\norm{F_{N,k}}_{ B_{k,D''}^{\rho,\gamma}\intervalcc{-T}{T}}\,.
\end{equation*}
\end{lem}
\begin{proof}
We need to estimate
\[
\sum_MC_{N,D'}(M)M^{\sigma'}
 \norm{P_M F_{N,k}}_{L_t^8(\intervalcc{-T}{T};L_x^4L_y^r(\R^2))}
\,.
\]
Given a dyadic integer $M\gg1$  and $0\leq t\leq T$, the Duhamel integral formulation of the truncated equation~\eqref{eq:trun-Fn} reads
\[
P_MF_{N,k}(t) = \e^{it\A}P_MP_{1,k}f_0 + i\int_0^t \e^{i(t-\tau)\A}P_M\parent{\mathcal{N}(F_{N,k},\phi,\phi)}(\tau)\dd\tau\,.
\]
First, we estimate the linear evolution of the unit-scale block $P_MP_{1,k}f_0$, which is nonzero only when $M\sim N$. To do so, we use the Bernstein-Strichartz estimate~\eqref{eq:str} from Lemma~\ref{lem:str} with $(p,q)=(8,4)$ and $\abs{E}\lesssim 1$. There holds
\[
\norm{P_M\e^{it\A}P_{1,k}f_0}_{L_t^8L_x^4L_y^r} \lesssim \mathbf{1}_{M\sim N}\norm{P_{1,k}f_0}_{L_{x,y}^2}\,.
\]

To handle the nonlinear part, we perform a Littlewood-Paley decomposition of $\phi$:
\[
\phi\sim\sum_{1\leq L_\leq N^\gamma}P_L\phi\,,
\]
and a Littlewood-Paley decomposition of $F_{N,k}$ recentered around $k$:
\[
F_{N,k}\sim \sum_{1\leq K} P_{K,k}F_{N,k}\,.
\]
In what follows, without loss of generality, we assume that $L_1\geq L_2$. The frequency support in the variable $\eta$ of the Duhamel term has length 
\[
\abs{\supp_{\eta} \mathcal{F}_{y\to\eta}\parent{\mathcal{N}(P_{K,k}F_{N,k} ,P_{L_1}\phi, P_{L_2}\phi )}} \lesssim \max(L_1,K)\,.
\]
We use the Bernstein-Strichartz $TT^*$-estimate~\eqref{eq:TTstar} from Lemma~\ref{lem:str} with $(p,q)=(8,4)$, $(\tilde p,\tilde q)=(4,\infty)$ and $\abs{E}\lesssim \max(L_1,K)$, and get 
\begin{multline*}
\sum_{L_2\leq L_1\leq N^\gamma}\sum_{K}\norm{\int_{0}^t e^{i(t-\tau)\A} P_M\parent{\mathcal{N}(P_{K,k}F_{N,k},P_{L_1}\phi,P_{L_2}\phi)}(\tau)\dd\tau}_{L_t^{8}L_x^4L_y^r}\\
	\lesssim \sum_{L_2\leq L_1\leq N^\gamma}\sum_{K}\max(L_1,K)^{\frac{1}{2}-\frac{1}{r}}\norm{ P_M\parent{\mathcal{N}(P_{K,k}F_{N,k},P_{L_1}\phi,P_{L_2}\phi)}(\tau)\dd\tau}_{L_t^\frac{4}{3}L_x^1L_y^2}\,.
\end{multline*}
Let us first consider the contribution of the terms with $K\ll N$. In such a case, only the terms with $M\sim N$ contribute. Moreover, using the truncation~\eqref{eq:trun-Fnk}, so that $\norm{\japbrak{D_y}^\sigma \phi}_{L_t^8L_x^4L_y^\infty}\lesssim1$, 
\begin{multline*}
    M^{\sigma'}\sum_{L_2\leq L_1\leq N^\gamma}\sum_{K\ll N}\max(L_1,K)^{\frac{1}{2}-\delta}\norm{P_M\parent{\mathcal{N}(P_{K,k}F_{N,k},P_{L_1}\phi,P_{L_2}\phi)}}_{L_t^\frac{4}{3}L_x^1L_y^2}\\
    \lesssim \un_{M\sim N} N^{\sigma'}\sum_{L_2\leq L_1\leq N^\gamma}\sum_{K\ll N}\max(L_1,K)^{\frac{1}{2}-\delta}K^{-\rho}\max\parentbig{1,\frac{K}{N^\gamma}}^{-D''}L_1^{-\sigma}L_2^{-\sigma}\\c_{k,D''}^{\rho,\gamma}(K)\norm{P_{K,k}F_{N,k}}_{L_t^\infty L_{x,y}^2}\,.
    \end{multline*}
When $K\leq N^\gamma$ we get by comparing $K$ and $L_1$ that
\begin{multline*}
    \un_{M\sim N}  N^{\sigma'}\sum_{K,L_1,L_2\leq N^\gamma}\max(L_1,K)^{\frac{1}{2}-\frac{1}{r}}K^{-\rho}L_1^{-\sigma}L_2^{-\sigma} \norm{ F_{N,k}}_{B_{k,D''}^{\rho,\gamma}}\\
    \lesssim \un_{M\sim N}N^{\sigma'+\gamma(\frac{1}{2}-\frac{1}{r}-\sigma) +\delta} \norm{ F_{N,k}}_{B_{k,D''}^{\rho,\gamma}}\,.
\end{multline*}
On the other hand, when $K\geq N^{\gamma}$, taking $D''\gg \frac{1}{2}$ yields
\[
\sum_{N^\gamma\leq K \ll N} K^{\frac{1}{2}-\frac{1}{r}-\rho}K^{-D''} \lesssim N^{\gamma(\frac{1}{2}-\frac{1}{r}-\rho -D'')}\,,
\]
so that
\begin{multline*}
    \un_{M\sim N}  N^{\sigma'}\left( \sum_{N^\gamma\leq K\ll N}\sum_{L_1,L_2\leq N^\gamma}K^{\frac{1}{2}-\frac{1}{r}-\rho}\parentbig{\frac{K}{N^\gamma}}^{-D''}L_1^{-\sigma}L_2^{-\sigma} \right) \norm{F_{N,k}}_{B_{k,D''}^{\rho,\gamma}} \\
    \lesssim \un_{M\sim N}  N^{\sigma'+\gamma(\frac{1}{2}-\frac{1}{r}-\sigma)+\delta}\norm{F_{N,k}}_{B_{k,D''}^{\rho,\gamma}}\,.
\end{multline*}
Summing over $M\sim N$, in which case $C_{N,D'}(M)\sim 1$, and taking the $L^{\frac 43}$-norm in time for $T\leq 1$, we obtain the desired bound for the contributions with $K\ll N$:
\begin{multline*}
    \sum_{M\sim N}C_{N,D'}(M) M^{\sigma'}\sum_{L_1,L_2\leq N^\gamma}\sum_{K\ll N}\max(L_1,K)^{\frac{1}{2}-\delta}\norm{P_M\parent{\mathcal{N}(P_{K,k}F_{N,k},P_{L_1}\phi,P_{L_2}\phi)}}_{L_t^\frac{4}{3}L_x^1L_y^2}\\
    \lesssim T^\frac{1}{2}N^{\sigma'+\gamma(\frac{1}{2}-\frac{1}{r}-\sigma) +\delta}\norm{F_{N,k}}_{L_t^\infty B_{k,D''}^{\rho,\gamma}}\,.
\end{multline*}

Next, we handle the contribution of the terms with $K\gtrsim N$, where only the terms with $K\gtrsim \max(N,M)$ actually contribute. We shall use the localization around frequency $K\sim N^\gamma $ of $F_{N,k}$ imposed by the norm $B_{k,D''}^{\rho,\gamma}$ to prove smallness: for sufficiently large $D''$ depending on $D'$, there holds 
\begin{multline*}
        M^{\sigma'}\sum_{L_2\leq L_1\leq N^{\gamma}}\sum_{K\gtrsim N}K^{\frac{1}{2}-\frac{1}{r}}\norm{P_M\parent{\mathcal{N}(P_{K,k}F_{N,k},P_{L_1}\phi,P_{L_2}\phi)}}_{L_t^\frac{4}{3}L_x^1L_y^2}\\
	\lesssim T^\frac{1}{2}M^{\sigma'}\sum_{K\gtrsim \max(N,M)}K^{\frac{1}{2}-\frac{1}{r}}\norm{P_{K,k}F_{N,k}}_{L_t^\infty L_{x,y}^2} \\
	\lesssim T^\frac{1}{2}M^{\sigma'}\sum_{K\gtrsim \max(N,M)}K^{\frac{1}{2}-\frac{1}{r}}K^{-\rho}\parentbig{\frac{K}{N^\gamma}}^{-D''}\norm{F_{N,k}}_{B_{k,D''}^{\rho,\gamma}}
\\
\lesssim T^\frac{1}{2}\parent{NM}^{-10(D'+1)}\norm{F_{N,k}}_{ B_{k,D''}^{\rho,\gamma}}\,.
\end{multline*}
Hence, the weight  $C_{N,D'}(M)=\max(\frac{N}{M},\frac{M}{N})^{D'}$ and the term $M^{\sigma'}$ are absorbed by the factor $\parent{NM}^{-10(D'+1)}$: we can sum over $M$ to get 
\begin{multline*}
    \sum_M C_{N,D'}(M)M^{\sigma'}\sum_{L_2\leq L_1\leq N^\gamma}\sum_{K\gtrsim N}K^{\frac{1}{2}-\frac{1}{r}}\norm{ P_M\parent{\mathcal{N}(P_{K,k}F_{N,k},P_{L_1}\phi,P_{L_2}\phi)}}_{L_t^\frac{4}{3}L_x^1L_y^2} \\
    \lesssim T^\frac{1}{2}\norm{ F_{N,k}}_{ B_{k,D''}^{\rho,\gamma}}\,.
\end{multline*}
This concludes the proof of the refined Strichartz estimates for the unit blocks $F_{N,k}$.
\end{proof}

Next, we infer probabilistic Strichartz estimates on the whole adapted linear evolution $F_N$ from the above estimate on the individual blocks $F_{N,k}$, combined with probabilistic decoupling (we recall that the high frequencies of the initial data are independent of the low frequencies). The proof follows the same lines as in~\cite{bringmann2021-ansatz}, Proposition 4.4.
\begin{prop}[Probabilistic Strichartz estimates]
\label{prop:srt-proba}
Let $F_N$ be a solution to the truncated equation~\eqref{eq:trun-Fn}, $s>0$, $D'>0$ and $T_0$ be as in Lemmas~\ref{lem:loc-Fnk} and~\ref{lem:loc-Fn}. There holds that for all $0<T\leq T_0$, $p\geq1$ and $\sigma'>0$,
\begin{equation}
\label{eq:str-rand}
    \norm{\japbrak{D_y}^{\sigma'}F_N}_{L_\omega^p(\Omega,S_{N,D'}\intervalcc{-T}{T})}
	+\lesssim_{\delta}  \sqrt{p}T^{\frac{1}{2}}N^{\sigma'+\frac{\gamma}{2}-s-\gamma\sigma+2\delta}\norm{\widetilde{P_N}f_0}_{L_x^2H_y^s}\,,
\end{equation}
where $\delta= \gamma(\sigma-\rho)>0$.
\end{prop}
The contributions to the power of $N$ can be explained as follows:
\[
N^{\sigma'-s+\frac{\gamma}{2}-\gamma\sigma}=\underbrace{N^{\sigma'-s}}_{\substack{\text{difference}  \\ \text{of derivatives}}}\cdot \underbrace{N^\frac{\gamma}{2}}_{\substack{\text{Strichartz with freq.} \\ \text{localization }}}\cdot\underbrace{N^{-\gamma\sigma}}_{\substack{\text{gain frequency} \\ \text{ localization}}}\,. 
\]
\begin{proof}
The idea is to use the independence between $\{g_l\colon \frac{N}{2}\leq |l|<N\}$ and $F_{N,k}$ which is $\mathcal{B}_{\frac N2}=\sigma(g_l\colon |l|<\frac{N}{2})$-measurable thanks to the assumptions on $\phi$. In order to use Minkowski's inequality in the large deviation estimates, we first prove the result with the norm $S_{N,D'}^{r}$ with $r\gg1$ instead of $S_{N,D'}$ (which denotes $S_{N,D'}^\infty$). We condition on $\mathcal{B}_{\frac N2}$, take $p\geq r$ and deduce from Minkowski's inequality that
\begin{equation*}
\label{eq:cond-kin}
    \begin{split}
        \norm{\japbrak{D_y}^{\sigma'}F_N}_{L^p_\omega S_{N,D'}^r}
	&=\mathbb{E}\Big[\mathbb{E}\Big[\norm{C_{N,D'}(M)\japbrak{D_y}^{\sigma'}P_MF_N}_{\ell_M^1L_t^8L_x^4L_y^r}^p\Big|\mathcal{B}_{\frac N2}\Big]\Big]^{\frac 1p}\\
	&\leq \mathbb{E}\Big[\Big\|\mathbb{E}\Big[\big|C_{N,D'}(M)\japbrak{D_y}^{\sigma'}P_MF_N\big|^p\Big|\mathcal{B}_{\frac N2}\Big]^{\frac 1p}\Big\|_{\ell_M^1L_t^8L_x^4L_y^r}^p\Big]^{\frac 1p}\,.
    \end{split}
\end{equation*}
Since 
\[
F_N=\sum_{\frac{N}{2}\leq |k|<N}g_k(\omega)F_{N,k}\,,
\]
The conditional Khintchine's inequality and the mutual independence between the $(g_k)$'s and the $(F_{N,k})$'s yield
\begin{equation*}
\mathbb{E}\left[\abs{\japbrak{D_y}^{\sigma'}P_MF_N}^p\Big|\mathcal{B}_{\frac N2}\right]^\frac{1}{p}
	\leq \sqrt{p}\,\norm{\mathbf{1}_{\frac{N}{2}\leq \abs{k}<N}\japbrak{D_y}^{\sigma'}P_MF_{N,k}}_{\ell_k^2}\,.
\end{equation*}
Plugging into~\eqref{eq:cond-kin}, 
\[
\norm{\japbrak{D_y}^{\sigma'}F_N}_{L^p_\omega S_{N,D'}^r}\lesssim\norm{\mathbf{1}_{\frac{N}{2}\leq \abs{k}<N}C_{N,D'}(M)\japbrak{D_y}^{\sigma'}P_MF_{N,k}}_{L_\omega^p \ell^1_M L_t^8L_x^4L_y^r\ell_k^2}\,.
\]
Thanks to the Cauchy-Schwarz inequality, we can go from $\ell^1_M$ to $\ell^2_M$ up to replacing the weight $C_{N,D'}$ by $C_{N,D'+1}$.
Using the Minkowski inequality again, we obtain
\begin{equation*}
    \begin{split}
        \norm{\japbrak{D_y}^{\sigma'}F_N}_{L^p_\omega S_{N,D'}^r}
	&\lesssim \sqrt{p}\,\norm{C_{N,D'+1}(M)\japbrak{D_y}^{\sigma'}P_MF_{N,k}}_{L_\omega^p\ell_M^1L_t^8L_x^4L_y^r\ell^2_k}\\
	&\lesssim\sqrt{p}\,\norm{C_{N,D'+1}(M)\japbrak{D_y}^{\sigma'}P_MF_{N,k}}_{\ell^2_kL_\omega^p\ell_M^2L_t^8L_x^4L_y^r}\\
	&\lesssim\sqrt{p}\,\norm{C_{N,D'+1}(M)\japbrak{D_y}^{\sigma'}P_MF_{N,k}}_{\ell^2_kL_\omega^pS_{N,D'+1}^r}\,.
    \end{split}
\end{equation*}
In the last estimate, we used that $\ell_M^1\subset\ell_M^2$. Finally, we apply the refined Strichartz estimate for the individual blocks $F_{N,k}$ (Lemma~\ref{lem:strichartz_block}). The frequency localization stated in Lemma~\ref{lem:loc-Fnk} yields 
\begin{align*}
\norm{\japbrak{D_y}^{\sigma'}F_N}_{L^p_\omega S_{N,D'}^r}
	&\lesssim  \sqrt{p} \norm{\japbrak{D_y}^{\sigma'}P_{1,k}f_0}_{\ell^2_k }
    +  \sqrt{p} T^{\frac{1}{2}}N^{\sigma'+\gamma\parent{\frac{1}{2}-\frac{1}{r}}-\gamma\sigma+\delta}\norm{ F_{N,k}}_{\ell^2_k  B_{k,D''}^{\rho,\gamma}}\\
    &\lesssim\sqrt{p}\, \norm{\widetilde{P}_Nf_0}_{L_x^2H_y^{\sigma'}}+\sqrt{p}T^{\frac{1}{2}}N^{\sigma'+\gamma\parent{\frac{1}{2}-\frac{1}{r}}-\gamma\sigma-s+\delta}\norm{\widetilde{P}_Nf_0}_{L_x^2H_y^s}\,.
\end{align*}
Meanwhile, the estimate for $p\leq \min(4,r)$ follows from Hölder's inequality. Next, we pass from $S_{N,D'}$ to $S_{N,D'}^r$ with $r\gg1$ by using the Sobolev embedding in $y$, and by loosing therefore an arbitrarily small power of $N^\frac{1}{r}$ (see Step 3 in the proof of Proposition 4.4 from~\cite{bringmann2021-ansatz}). Here we can choose $N^{\delta}$.
\end{proof}

\section{Estimates on the nonlinear term \texorpdfstring{$w_N$}{wn}}\label{sec:nonlinear_evolution}
For fixed $N$, we construct the local solution to the truncated equation on a time interval $[-T_0, T_0]$ that does not depend on $N$.
\begin{prop}
\label{prop:w}
Assume that  $0<\sigma < \nu-\frac{1}{2}$, $0<\sigma<\sigma'<s$ and $0<\gamma<1$. Fix $\alpha>0$ such that $\max(\nu-\sigma',\sigma)<\alpha<\nu$. Let $D= D(s,\nu,\sigma',\sigma,\alpha)>0$ and $D'\geq D'(s,\nu,\sigma',\sigma,\alpha,D)>0$ be sufficiently large. There exist $T_0=T_0(s,\nu,\sigma',\sigma,\alpha,D,D')>0$ smaller than the $T_0$ in Lemmas~\ref{lem:loc-Fnk} and~\ref{lem:loc-Fn}, and a unique solution $w_N\in Y_N^\nu(\intervalcc{-T_0}{T_0})$ to the truncated equation~\eqref{eq:wtheta}. In addition, for all $0<T<T_0$,
\begin{equation}
    \label{eq:a-priori}
    \norm{w_N}_{Y_N^\nu(\intervalcc{-T}{T})}\lesssim_\delta T^\frac{1}{2}N^{\nu-\sigma'-\gamma\nu}\norm{\japbrak{D_y}^{\sigma'}F_N}_{S_{N,D'}}+T^\frac{1}{2}N^{\nu-s-\gamma\sigma'}\norm{\japbrak{D_y}^sF_N}_{X_{N,D'}}\,.
\end{equation}
\end{prop}
To prove such a proposition we do a contraction mapping argument, which is a consequence of the a priori estimates stated below.

\subsection{A priori trilinear estimates}
\label{sec:a-priori}
We denote the Duhamel integral by
\[
\mathcal{I}f(t):=\int_0^t e^{i(t-\tau)\A}f(\tau)\dd\tau\,,
\]
and we write
\begin{equation}
    \label{eq:gamma}
    \Gamma w(t):=\mathcal{I}\widetilde{\mathcal{N}}_{\theta}(w)(t)\,,
\end{equation}
where we recall that $\widetilde{\mathcal{N}}_\theta$ was defined in~\eqref{eq:Ntilde}. Applying the trilinear estimates from Section~\ref{section:trilinear_proof}, together with the truncation of the nonlinearity, we obtain the following a priori estimates on the nonlinear component $w_N$.
\begin{prop}[A priori estimate on $w_N$]
Under the same constraints on the parameters as in the statement of Proposition~\ref{prop:w}, for $0\leq T\leq T_0$, there holds
\[
\norm{w_N}_{Y_N^\nu(\intervalcc{-T}{T})}\lesssim T^\frac{1}{2}\norm{w_N}_{Y_N^\nu(\intervalcc{-T}{T})}+T^\frac{1}{2}N^{\nu-\sigma'-\gamma\nu}\norm{\japbrak{D_y}^{\sigma'}F_N}_{S_{N,D'}}+T^\frac{1}{2}N^{\nu-s-\gamma\sigma'}\norm{\japbrak{D_y}^sF_N}_{X_{N,D'}}\,.
\]
\end{prop}
Let us briefly comment the proof of the above estimate. Notice that the new terms of the approximate solution $u_N$ constructed at step $N$, which are either $F:=F_N$ or $w:=w_{n}$, appear at least once in the nonlinear forcing term (because $\mathcal{N}(u_{\frac N2})$ is removed from the nonlinearity). 

Then, the strategy to obtain optimal a priori estimates on $w_N$ is to use the $TT^\ast$ estimate from Lemma~\ref{lem:str} and the refined probabilistic estimate from Proposition~\ref{prop:srt-proba}. Specifically, when we have trilinear mixed terms (i.e some functions of type $F_N$ interact with others of type $w_N$), we place $F_N$ in $L_y^\infty$ and $w_{n}$ in $L_y^2$, respectively. Surprisingly, we do so even if $F_N$ has the highest frequency, in order to reduce the loss of derivatives from the \emph{deterministic} Strichartz estimates of Lemma
~\ref{lem:str}. This is implemented precisely in the paracontrolled trilinear estimates collected in Section~\ref{section:trilinear_proof}.

\begin{proof}
Up to permutation, we always assume that a term of type $w_N$ or $F_N$ is at the first position. We collect the different contributions depending on the nature of the terms. We place the terms $u_{N/2}$ (denoted $\phi$ in Section~\ref{section:trilinear_proof}) in the last position. Thanks to the truncation, they satisfy $\|\langle D_y\rangle^\sigma \phi\|_{S_{\leq N,D}}\lesssim 1$.

~\\
\emph{Contribution of $F_N F_N F_N$.} We have from~\eqref{eq:ff-hl}, \eqref{eq:fg}, \eqref{eq:fg-lhh} and~\eqref{eq:fg-hh} that
    \[
    \norm{\mathcal{I}\parent{\theta_{F;N}^2\mathcal{N}\parent{F_N, F_N, F_N}}}_{Y_N^\nu} \lesssim T^\frac{1}{2}N^{\nu-s-\sigma'}\norm{\japbrak{D_y}^sF_N}_{X_{N,D'}}\,.
    \]
~\\
\emph{Contribution of $w_N F_N F_N$.} We have from~\eqref{eq:wg-hl} and \eqref{eq:wg}, \eqref{eq:wg-lhh}, \eqref{eq:wg-hh} that
    \[
    \norm{\mathcal{I}\parent{\theta_{F;N}^2\mathcal{N}\parent{w_N,F_N ,F_N}}}_{Y_N^\nu} \lesssim T^\frac{1}{2}\norm{w_N}_{Y_N^\nu}\,.
    \]
~\\  
\emph{Contribution of $w_N w_N F_N$.}
We have  from~\eqref{eq:wv-hl}, \eqref{eq:wv}, \eqref{eq:wv-lhh} and~\eqref{eq:wv-hh} that
    \[
    \norm{\mathcal{I}\parent{\theta_{w;N}\theta_{F;N}\mathcal{N}(w_N, w_N, F_N)}}_{Y_N^\nu} \lesssim T^\frac{1}{2}\norm{w_N}_{Y_N^\nu}\,.
    \]
~\\
\emph{Contribution of $w_N w_N w_N$.}  We have from~\eqref{eq:wv-hl}, \eqref{eq:wv}, \eqref{eq:wv-lhh} and \eqref{eq:wv-hh} that
    \[
    \norm{\mathcal{I}\parent{\theta_{w;N}^2\mathcal{N}\parent{w_N,w_N,w_N}}}_{Y_N^\nu} \lesssim T^\frac{1}{2}\norm{w_N}_{Y_N^\nu}\,.
    \]
~\\  
\emph{Contribution of $w_N u_{\frac N2} u_{\frac N2}$.} 
We have from~\eqref{eq:wg-hl}, \eqref{eq:wg}, \eqref{eq:wg-lhh}, \eqref{eq:wg-hh} or from \eqref{eq:wv-hl}, \eqref{eq:wv}, \eqref{eq:wv-lhh}, \eqref{eq:wv-hh} depending on the nature of the first factor $u_{\frac N2}$ that
    \[
    \norm{\mathcal{I}\parent{\theta_{F,w;\leq \frac N2}\mathcal{N}\parent{w_N,u_{\frac N2},u_{\frac N2}}}}_{Y_N^\nu} \lesssim T^\frac{1}{2}\norm{w_N}_{Y_N^\nu}\,.
    \]
~\\
\emph{Contribution of $w_N w_N u_{\frac N2}$.}
We have from~\eqref{eq:wv-hl}, \eqref{eq:wg}, \eqref{eq:wg-lhh} and \eqref{eq:wg-hh} that
\[
    \norm{\mathcal{I}\parent{\theta_{w;N}\theta_{F,w;\leq \frac N2}\mathcal{N}\parent{w_N, w_N, u_{\frac N2}}}}_{Y_N^\nu} \lesssim T^\frac{1}{2}\norm{w_N}_{Y_N^\nu}\,.
    \]
~\\  
\emph{Contribution of $w_N F_N^{\omega} u_{\frac N2}$.}  We have from~\eqref{eq:wg-hl}, \eqref{eq:wg}, \eqref{eq:wg-lhh} and \eqref{eq:wg-hh} that 
    \[
    \norm{\mathcal{I}\parent{\theta_{F;N}\theta_{F,w;\leq \frac N2}\mathcal{N}\parent{w_N,F_N,u_{\frac N2}}}}_{Y_N^\nu} \lesssim T^\frac{1}{2}\norm{w_N}_{Y_N^\nu}\,.
    \]  
~\\
\emph{Contribution of $F_N F_N^{\omega} u_{\frac N2}$.} We have from~\eqref{eq:ff-hl}, \eqref{eq:fg}, \eqref{eq:fg-lhh} and~\eqref{eq:fg-hh} that
    \[
    \norm{\mathcal{I}\parent{\theta_{F;N}\theta_{F,w;\leq \frac N2}\mathcal{N}\parent{F_N, F_N ,u_{\frac N2}}}}_{Y_N^\nu} \lesssim T^\frac{1}{2}N^{\nu-s-\sigma'}\norm{\japbrak{D_y}^sF_N}_{X_{N,D'}}\,.
    \]
~\\
\emph{Contribution of $F_N u_{ n-1}u_{\frac N2}$}. It follows from the definition of the adapted linear evolution $F_N$ that these terms only contribute to the equation satisfied by $w_N$ when one of the two terms $u_{N/2}$ is truncated at frequencies $>N^\gamma$. We split this term into
    \[
    P_{>N^\gamma}u_{\frac N2} = P_{>N^\gamma} w_{\leq \frac N2} + P_{>N^\gamma} F_{\leq \frac N2}\,.
    \]
There holds
    \[
    \norm{\mathcal{I}\parent{\theta_{F,w;\leq \frac N2}^2\mathcal{N}\parent{F_N,P_{>N^\gamma}w_{\leq \frac N2},u_{\frac N2}}}}_{Y_N^\nu}\lesssim T^\frac{1}{2}N^{\nu-\sigma'-\gamma\nu}\norm{\japbrak{D_y}^{\sigma'}F_N}_{S_{N,D'}}\,.
    \]
    Indeed, we use~\eqref{eq:fv-gam-hl} for the \emph{high-low-low}-type interactions. For the other interactions, we use~\eqref{eq:fg}, \eqref{eq:fg-lhh}, \eqref{eq:fg-hh} when $u_{\frac N2}$ is of type $G$, and~\eqref{eq:fv}, \eqref{eq:fv-lhh}, \eqref{eq:fv-hh} when $u_{\frac N2}$ is of type $v$, respectively.
    Similarly, we have
    \[
    \norm{\mathcal{I}\parent{\theta_{F,w;\leq \frac N2}^2\mathcal{N}\parent{F_N,P_{>N^\gamma}F_{\leq \frac N2}, u_{\frac N2}}}}_{Y_N^\nu}\lesssim T^\frac{1}{2}N^{\nu-s-\gamma\sigma'}\norm{\japbrak{D_y}^sF_N}_{X_{N,D'}}\,,
    \]
    which follows from~\eqref{eq:fg-gam-hl} for the \emph{high-low-low} type interactions, and ~\eqref{eq:fg}, \eqref{eq:fg-lhh}, \eqref{eq:fg-hh}, or \eqref{eq:fv}, \eqref{eq:fv-lhh}, ~\eqref{eq:fv-hh} depending on the nature of $u_{\frac N2}$.
\end{proof}

\subsection{Contraction mapping argument}

As a consequence of the a priori estimates~\eqref{eq:a-priori}, we see that $\Gamma$ (defined in~\eqref{eq:gamma}) maps $Y_N^\nu$ into $Y_N^\nu$. We now prove that it is a contraction mapping in the Banach space $Y_N^\nu$ i.e. that for all $z,w\in Y_N^\nu[-T,T]$ when $T\leq T_0$,
\[
\norm{\Gamma(z)-\Gamma(w)}_{Y_N^\nu}\leq \frac{1}{2}\norm{z-w}_{Y_N^\nu}\,.
\]
To achieve this goal, we adapt the proof in~\cite{bringmann2021-ansatz}, inspired by~\cite{DeBouardDebussche99}, to the trilinear case . Since the linear terms are handled as above, it only remains to consider the quadratic and cubic terms. Let 
\[
t_z := \sup\set{0\leq t\leq T\;\colon  \|z\|_{Y_N^{\nu}[-t,t]}
	\leq 2}\,.
\]
The time $t_{w}$ is defined similarly. We observe from a continuity argument (both on the norm and on $z,w$) that for all $t> t_z$ (resp. $t> t_{w}$), we have $\theta_{z;N}(t)=0$ (resp. $\theta_{w;N}(t)=0$). Without loss of generality, we assume that $t_z\leq t_{w}$ and we encapsulate the nonlinear terms under consideration in $\Lambda$:
\begin{equation*}
    \begin{split}
    \Lambda(z,w)(t)&= \Lambda_1(z,w)(t) + \Lambda_2(z,w)(t) + \Lambda_3(z,w)(t)\,,\\
      \Lambda_1(z,w)(t) &=\theta_{z;N}(t)^2\abs{z}^2z(t)-\theta_{w;N}(t)^2\abs{w}^2w(t) \,,\\
\Lambda_2(z,w)(t)  &=\theta_{F;N}(t)\parentbig{\theta_{z;N}(t)\mathcal{N}(z,z,F_N)(t)-\theta_{w;N}(t)\mathcal{N}(w,w,F_N)(t)}\,, \\
\Lambda_3(z,w)(t)&=\theta_{F,w;\leq \frac N2}(t)\parentbig{\theta_{z;N}(t)\mathcal{N}(z,z,u_{\frac N2})(t)-\theta_{w;N}(t)\mathcal{N}(w,w,u_{\frac N2})(t)}\,.
    \end{split}
\end{equation*}
Let us consider the cubic term $\Lambda_1(z,w)$. We decompose the nonlinear interactions as in~\cite{DeBouardDebussche99}: 
\begin{equation*}
    \begin{split}
    \normbig{\int_0^t\e^{i(t-\tau)\A}\Lambda_1(\tau)\dd\tau}_{Y_N^\nu} 
    &\leq \normbig{\int_0^t\e^{i(t-\tau)\A}\mathbf{1}_{\intervalcc{0}{t_z}}(\tau)\parent{\theta_{z;N}(\tau)-\theta_{w;N}(\tau)}\abs{z}^2z(\tau)\dd\tau}_{Y_N^\nu} \\
    &+ \normbig{\int_0^t\e^{i(t-\tau)\A}\mathbf{1}_{\intervalcc{0}{t_z}}(\tau)\theta_{w;N}(\tau)\parent{\abs{z}^2z(\tau)-\abs{w}^2w(\tau)}\dd\tau}_{Y_N^\nu} \\
    &+ \normbig{\int_0^t\e^{i(t-\tau)\A}\mathbf{1}_{\intervaloc{t_z}{t_{w}}}(\tau)\parent{\theta_{z;N}(\tau)-\theta_{w;N}(\tau)}\abs{w}^2w(\tau)\dd\tau}_{Y_N^\nu}\,.
    \end{split}
\end{equation*}
Here, we used that $\theta_{z;N}(\tau)=0$ when $\tau\in\intervaloc{t_z}{t_{w}}$. Then, we deduce from the multilinear estimates ~\eqref{eq:wv-hl}, \eqref{eq:wv}, \eqref{eq:wv-lhh} and \eqref{eq:wv-hh} that
\begin{equation*}
    \begin{split}
    \normbig{\int_0^t\e^{i(t-\tau)\A}\Lambda_1(\tau)\dd\tau}_{Y_N^\nu} 
    & \lesssim  T^\frac{1}{2}\norm{\theta_{z;N}-\theta_{w;N}}_{L^\infty_t}\norm{\mathbf{1}_{\intervalcc{0}{t_z}}z}_{Y_N^\nu}^3 \\
    &+T^\frac{1}{2}\parent{\norm{\mathbf{1}_{\intervalcc{0}{t_z}}z}_{Y_N^\nu}^2+\norm{\mathbf{1}_{\intervalcc{0}{t_z}}w}_{Y_N^\nu}^2}\norm{z-w}_{Y_N^\nu}\\
    &+T^\frac{1}{2}\norm{\theta_{z;N}-\theta_{w;N}}_{L^\infty_t}\norm{\mathbf{1}_{\intervalcc{t_z}{t_{w}}}w}_{Y_N^\nu}^3\\
    &\lesssim T^\frac{1}{2}\norm{z-w}_{Y_N^\nu}\,.
    \end{split}
\end{equation*}
The quadratic interactions are controlled analogously. Hence, $\Gamma$ is a contraction mapping on $Y_N^\nu$ for sufficiently small $T_0$, and this concludes the proof of Proposition~\ref{prop:w}. 

\section{Conclusion}\label{sec:conclusion}
We now show that the sequence $(u_N)_N$ tends to a limit $u$ on some time interval $[-T_0,T_0]$, that $u$ is solution to equation~\eqref{eq:NLSHW} on a random time interval $[-T,T]$, and therefore finish the proof of Theorem~\ref{theo:main}.
We refer to~\cite{bringmann2021-ansatz}, Appendix A regarding the measurability properties of the functions $F_N$ and $w_N$.

\subsection{Optimizing the constraints}

We collect the main constraints resulting from the nonlinear analysis.
\begin{itemize}
    \item Constraint coming from the adapted linear evolution $F_N$ in Strichartz type spaces $S_{N,D'}$:
\begin{equation*}
    \sigma' + \frac{\gamma}{2} - s -\gamma\sigma <0\,.
\end{equation*}
    \item Contribution of term $F_N (P_{> N^\gamma}F_{\leq \frac N2}) u_{\frac N2}$ in $Y_N^\nu$: 
    \begin{equation*}
        \nu-s-\gamma\sigma'<0\,.
    \end{equation*}
    \item Contribution of term $F_N (P_{> N^\gamma}w_{\leq \frac N2}) u_{\frac N2}$ in $Y_N^\nu$:
    \begin{equation*}
 \nu-\sigma'-\gamma\sigma<0\,.
    \end{equation*}
\end{itemize}
This reduces to the minimization problem for $s$ under the constraints  
\begin{equation*}
\begin{cases}
    \sigma'-s-\gamma\sigma + \frac{\gamma}{2}<0\\
    \nu - \gamma\sigma' -s <0\\
    \nu - \sigma' - \gamma\nu <0\\
    \sigma - \nu <-\frac{1}{2}\\
    \sigma-\sigma' < 0\,.
\end{cases}
\end{equation*}
Discretizing $\gamma\in\intervaloo{0}{1}$ and using a linear programming solver leads to an approximate optimal solution 
\begin{equation}
    \label{eq:opti}
    (s_0,\sigma'_0,\sigma_0,\nu_0) = (0.464131, 0.154581, 0.077290, 0.577291) \,,\quad \gamma =  0.732232\,.
\end{equation}
We chose $s>s_0$, $\sigma=\sigma_0-$, $\sigma'=\sigma'_0-$ and $\nu=\nu_0-$, and we now establish Theorem~\ref{theo:main}.

\subsection{Convergence of \texorpdfstring{$(u_N)_{N\in2^\N}$}{un}}

First, we prove the convergence~\eqref{eq:convergence} in expectation along the dyadic subsequence $(u_N)_{N\in2^\N}$ up to the time $T_0>0$ which is chosen so that Proposition~\ref{prop:w} holds true, and yields the a priori estimate on the truncated solution. 
In what follows $T$ is fixed, with $0<T\leq T_0$, and the norms are taken on the slab $\Omega\times\intervalcc{-T}{T}\times\R^2$. We recall from~\eqref{eq:un_series} that for some $N_0$ sufficiently large, the approximate solution reads
\[
u_N = u_{N_0} + \sum_{M=N_0}^N \left(w_M + F_M\right)\,,,
\]
We first prove the convergence of the series $\sum F_M$ and $\sum {w_M}$ in the space $L^2_\omega(\Omega ; X_T^{s,\sigma})$, where the Sobolev-Strichartz space $X_{T}^{s,\sigma}$ is defined in~\eqref{eq:XT0}. Thanks to a completeness argument, it is enough to prove that the partial sums form a Cauchy sequence. Let $0\leq N_-\leq N_+<\infty$. We first address the adapted linear evolution, controlled in $L^2_\omega(\Omega ;\mathcal{C}_t ([-T_0,T_0] ; L^2_xH^s_y))$:
\begin{multline*}
\normbig{\sum_{M=N_-}^{N_+}\japbrak{D_y}^sF_M}_{L^2_\omega L^{\infty}_t L^2_{x,y}}
	\lesssim\normbig{\sum_{M=N_-}^{N_+}\japbrak{D_y}^sP_NF_M}_{L^2_\omega L^{\infty}_t \ell^2_N L^2_{x,y}}
	\lesssim \normbig{\sum_{M=N_-}^{N_+}\japbrak{D_y}^sP_NF_M}_{L^2_\omega \ell^2_N L^{\infty}_t  L^2_{x,y}}\,.
\end{multline*}
From the definition~\eqref{eq:cND} of the spaces $X_{M,D'}$, we have
\begin{equation*}
\normbig{\sum_{m=N_-}^{N_+}\japbrak{D_y}^sF_M}_{L^2_\omega L^{\infty}_t L^2_{x,y}}
	\lesssim \normbig{\sum_{m=N_-}^{N_+}\max\left(\frac NM,\frac MN\right)^{-D'}\|\japbrak{D_y}^sF_M\|_{X_{M,D'}}}_{L^2_\omega \ell^2_N}\,.
\end{equation*}
Applying Cauchy-Schwarz's inequality, and then Lemma~\ref{lem:loc-Fn}, the right-hand-side of the above inequality is bounded by
\[
{\rm r.h.s}\
	\lesssim \normbig{\sum_{M=N_-}^{N_+}\max\left(\frac NM,\frac MN\right)^{-D'} \|\widetilde{P_M}f_0^{\omega}\|_{L^2_xH^s_y}}_{L^2_\omega \ell^2_N}
	\lesssim \left(\sum_{M=N_-}^{N_+} \|\widetilde{P_M}f_0\|_{L^2_xH^s_y} ^2\right)^{\frac 12}\,.
\]
Similarly, convergence in $L^2_\omega( \Omega ; L_t^8([-T_0 ; T_0] ; L_x^4W_y^{\sigma,\infty}))$ follows from the probabilistic Strichartz estimate in Proposition~\ref{prop:srt-proba}
\[
\norm{\japbrak{D_y}^{\sigma}F_M}_{L^2_\omega L_t^8L_x^4L_y^{\infty}}
	\lesssim  T^{\frac{1}{2}}M^{\sigma+\frac{\gamma}{2}-s-\gamma\sigma+2\delta}\norm{P_M f_0}_{\mathcal{H}^s}\,,
\]
where $\sigma+\frac{\gamma}{2}-s-\gamma\sigma+2\delta<0$ when $\rho$ is chosen close enough to $\sigma$ since $\sigma<\sigma'$. To establish the convergence of the nonlinear remainders $(w_N)_N$, we use the a priori estimate from Proposition~\ref{prop:w}. Namely,
\[
\norm{w_M}_{Y_M^\nu}
	\lesssim T^\frac{1}{2}M^{\nu-\sigma'-\gamma\nu}\norm{\japbrak{D_y}^{\sigma'}F_M}_{S_{M,D'}}+T^\frac{1}{2}M^{\nu-s-\gamma\sigma'}\norm{\japbrak{D_y}^sF_M}_{X_{M,D'}}\,,
\]
where $\nu-\sigma'-\gamma\nu<0$ and $\nu-s-\gamma\sigma'<0$. Using again Lemma~\ref{lem:loc-Fn} and Proposition~\ref{prop:srt-proba}, we get that for a very small parameter $\varepsilon>0$, 
\[
\norm{\japbrak{D_y}^\sigma w_M}_{L_\omega^2 L^8_tL_x^4L^\infty_y}+\norm{\japbrak{D_y}^\nu w_M}_{L_\omega^2 L^{\infty}_tL^{2}_{x,y}}
	\lesssim \norm{w_N}_{Y_M^\nu}\lesssim T^{\frac 12}M^{-\varepsilon}\|\widetilde{P_M}f_0\|_{L^2_xH^s_y}\,,
\]
so that the sum over $M$ converges.
\subsection{The limit is a strong solution}
We now show that almost-surely in $\omega$ there exists $0<T^\omega\leq T_0$ such that the limit $u$ satisfies the Duhamel formula for equation~\eqref{eq:NLSHW} on $[-T^\omega,T^\omega]$. To achieve such a goal, we prove that the truncation functions $\theta_{F,w;\leq \frac N2}$, $\theta_{F;N}$ and $\theta_{w;N}$ are actually equal to $1$ at least until a time $T^\omega>0$ almost-surely. First, we observe that the map
\[
t\in[-T_0,T_0]\mapsto \sum_{M\geq N_0}\|\japbrak{D_y}^{\sigma'}F_M\|_{S_{M,D'}([-t,t])}
\]
is almost-surely finite thanks to Proposition~\ref{prop:srt-proba}, in which case it is continuous. Moreover, it is equal to zero at $t=0$. Therefore, the random time $T_1(\omega)$, defined by
\[
T_1(\omega)=\sup\Bigg\{t\in[-T_0,T_0] \;\colon \sum_{M\geq 0}\|\japbrak{D_y}^{\sigma'}F_M\|_{S_{M,D'}([-t,t])}\leq \frac 12 \Bigg\}\,,
\]
is almost-surely positive. Then we deduce from Proposition~\ref{prop:w} that for a very small parameter $\varepsilon>0$, we have
\[
\norm{ w_M}_{Y_M^\nu([-t,t])}
	\lesssim T^\frac{1}{2}M^{-\varepsilon}\parent{\norm{\japbrak{D_y}^{\sigma'}F_M}_{S_{M,D'}([-t,t])}+\norm{\japbrak{D_y}^sF_M}_{X_{M,D'}([-t,t])}}\,,
\]
the series being almost surely finite at least until the time $T_0$ thanks to Lemma~\ref{lem:loc-Fn} and Proposition~\ref{prop:srt-proba}, and it is equal to zero at $t=0$. Therefore, the random time $T^\omega$ defined by
\[
T^\omega=\sup\Bigg\{t\in[-T_1(\omega),T_1(\omega)] \;\colon \sum_{M\geq N_0}\|w_M\|_{Y_M^{\nu}([-t,t])}\leq \frac 12 \Bigg\}\,
\]
 is almost-surely nonnegative. Therefore, on the time interval $[-T^\omega,T^\omega]$ one has $\theta_{F,w;\leq \frac N2}(t)=\theta_{F;N}(t)=\theta_{w;N}(t)=1$, so that for every $N$, $u_N$ is an actual solution to~\eqref{eq:NLSHW} with initial data $P_{\leq N}f_0^{\omega}$. We conclude by passing to the limit $N\to\infty$ in the Duhamel formula. 

\subsection{The \texorpdfstring{$H_x^{2s}L_y^2$ part}{x}}

We prove that for each $N\in2^\N$ the sequence $(u_N)$, a priori constructed in $\mathcal{C}(\intervalcc{-T_0}{T_0};L_x^2H_y^s)$, actually converges to $u$ in $\mathcal{C}(\intervalcc{-T_0}{T_0};\mathcal{H}^s)$. We simply use the $TT^*$-estimate to have a priori estimates of the $H_x^{2s}L_y^2$-norm of $u_N$:
\begin{equation*}
    \begin{split}
    \norm{u_N}_{L^\infty_t H_x^{2s}L_y^2} &\lesssim \norm{P_Nf_0^\omega}_{H_x^{2s}L_y^2} + \norm{\japbrak{D_x}^{2s}\abs{u_N}^2u_N}_{L_t^\frac{4}{3}(\intervalcc{-T_0}{T_0}; L_x^1L_y^2)} \\
&\lesssim \norm{P_Nf_0^\omega}_{H_x^{2s}L_y^2}+T_0^\frac{1}{2}\norm{u_N}_{L^\infty_t H_x^{2s}L_y^2}\norm{u_N}_{L_t^8L_x^4L_y^\infty}^2 \\
&\leq C\norm{P_Nf_0^\omega}_{H_x^{2s}L_y^2} + \frac{1}{2}\norm{u_N}_{L^\infty_t H_x^{2s}L_y^2}\,,
    \end{split}
\end{equation*}
for $T_0$ sufficiently small, and where we used the truncation argument of De Bouard-Debussche to control the $L_t^8L_x^4L_y^\infty$-norm of $u_N = u_{\frac N2} + F_N + w_N$. The rest of the argument follows as above. 

\subsection{Limit of the whole sequence} Once the sequence $(u_N)_{N\in2^\N}$ is constructed and its convergence established, we define the general approximating sequence $(u_n)_{n\in\N}$ by iteration: given $n\in\N$ and $N$ such that $\frac N2< n\leq N$, we set
\[
u_n = u_\frac{N}{2} + F_n + w_n\,,
\]
where the adapted linear evolution $F_n$ solves
\begin{equation*}
i\partial_t F_n + \mathcal{A}F_n = \mathcal{N}(F_n,P_{\leq N^\gamma}u_\frac{N}{2},P_{\leq N^\gamma}u_\frac{N}{2})\,,
\end{equation*}
with initial data
\[
F_n(0) = P_{\leq n}f_0^\omega - P_{\leq \frac{N}{2}}f_0^\omega = \sum_{\frac N2<\abs{k}\leq N}g_k(\omega)P_{1;k}f_0\,.
\]

Note that $(u_N)_N$ is a subsequence of $(u_n)_n$. The truncated versions $F_{n,\theta}$ and $w_{n,\theta}$ are defined as in~\eqref{eq:trun-Fn} and~\eqref{eq:wtheta}, respectively. For fixed $n\in\N$ and $N\in2^\N$ such that $\frac N2 < n \leq N$, we can prove the same estimates for $F_{n,\theta}$ and $w_{n,\theta}$ that we get for $F_{N,\theta}$ (in Section~\ref{sec:linear_evolution}) and $w_{N,\theta}$ (in Section~\ref{sec:nonlinear_evolution}). We deduce from these estimates and from the convergence of $(u_{N,\theta})_{N\in2^\N}$ that $(u_{n,\theta})_{n\in \N}$ is a Cauchy sequence. Finally we conclude that the whole sequence $(u_n)_n$ is convergent to the limit $u$ of $(u_N)_N$ in $X_{T_0}^{s,\sigma}$, moreover it is a solution to equation~\eqref{eq:NLSHW} with initial data $P_{\leq n}f_0^{\omega}$ on $[-T^{\omega},T^{\omega}]$ by construction.

\section{Paracontrolled trilinear estimates}\label{section:trilinear_proof}

In order to establish the trilinear estimates, we will  perform a Littlewood-Payley decomposition. After we cut in frequencies, we use in various situations the following decomposition.

\begin{lemma}\label{lem:PM}
For every $\phi^{(1)},\phi^{(2)},\phi^{(3)}$ such that the upper bound is finite, we have
\begin{multline*}
    \norm{\mathcal{I}\parent{ \mathcal{N} \parent{\phi^{(1)},\phi^{(2)},\phi^{(3)}}}}_{Y_N^\nu}\\ 
    \lesssim\sum_M \parent{C_{\leq N,D}(M)+C_{N,\alpha}(M)}M^{\nu}\norm{P_M\mathcal{N}\parent{\phi^{(1)},\phi^{(2)},\phi^{(3)}}}_{L_t^{\frac{4}{3}}L_x^1L_y^2}\,.
\end{multline*}
\end{lemma}
\begin{proof}[Proof of Lemma~\ref{lem:PM}]
This follows from the definition of the norms in Section~\ref{sec:scheme}. To control the Strichartz-norms we use the $TT^*$-Bernstein-Strichartz estimate from Lemma~\ref{lem:str} with $(\tilde p, \tilde q) = (4,\infty)$, and either  $(p,q)=(\infty,2)$ and $r=2$ or $(p,q)=(8,4)$ and $r=\infty$ respectively. We have the $y$-frequency localization on an interval $E=\set{\frac{M}{2}\leq \abs{\eta}\leq 2M}$ of length $\abs{E}\lesssim M$. Moreover, we use that $\nu=\sigma+\frac{1}{2}+0$ in the second case.
\end{proof}

We later use this Lemma for some given functions $\phi^{(1)},\phi^{(2)},\phi^{(3)}$ of type $F,G,w,v$, where
\[
F = F_N\,,\quad G = F_N\ \text{or}\ F_{\leq \frac N2}\,,\quad w = w_N\,,\quad v = w_N\ \text{or}\ w_{\leq \frac N2}\,.
\]
Next, we perform a Littlewood-Paley decomposition of each function, and we conduct a case-by-case analysis of each paracontrolled term defined below, depending on the nature of the functions $(\phi^{(i)})_{1\leq i\leq3}$ occurring in the equation~\eqref{eq:wtheta} satisfied by $w_{n }$. 
\begin{align*}
\Pi_{\hi,\lo,\lo}(\phi^{(1)},\phi^{(2)},\phi^{(3)}) &= \sum_{N_2,N_3\ll N_1}\mathcal{N}(P_{N_1}\phi^{(1)},P_{N_2}\phi^{(2)},P_{N_3}\phi^{(3)})\,,\\
\Pi_{\lo,\hi,\lo}(\phi^{(1)},\phi^{(2)},\phi^{(3)}) &= \sum_{N_1,N_3\ll N_2}\mathcal{N}(P_{N_1}\phi^{(1)},P_{N_2}\phi^{(2)},P_{N_3}\phi^{(3)})\,,\\
\Pi_{\lo,\hi,\hi}(\phi^{(1)},\phi^{(2)},\phi^{(3)}) &= \sum_{N_1\ll N_2\sim N_3 }\mathcal{N}(P_{N_1}\phi^{(1)},P_{N_2}\phi^{(2)},P_{N_3}\phi^{(3)})\,,\\
\Pi_{\hi,\hi}(\phi^{(1)},\phi^{(2)},\phi^{(3)}) &= \sum_{N_3\lesssim N_1\sim N_2}\mathcal{N}(P_{N_1}\phi^{(1)},P_{N_2}\phi^{(2)},P_{N_3}\phi^{(3)})\,.
\end{align*}

To estimate the norm $L_t^\frac{4}{3}L^1_xL^2_y$ of the Duhamel term, we use the Hölder inequality, putting one term in $L_t^\infty L_{x,y}^2$, the other two in $L_t^8L_x^4L_y^\infty$, and gaining a factor $T^\frac{1}{2}$. Recall that any small loss of derivative in $y$ is harmless, since our goal is to prove a result for some $s>s_0$. We multiply by $N^{0^+}$ to account for such a $0^+$-loss. At the end the loss are collected in the factor $N^\delta$, where $\delta>\sigma-\rho>0$ is arbitrarily small.

\begin{prop}[Trilinear estimates with \emph{high-low-low} type interactions]
\label{prop:hll}
Let $\phi$ be of type $v$ or $G$, and assume that $\max(\nu-\sigma',\sigma)<\alpha<\nu$,  $D>\alpha$ and $2D+s+\sigma'<D'$. Assuming that $\norm{\japbrak{D_y}^\sigma\phi}_{S_{\leq N,D}}\lesssim 1$, we have
\begin{align}
    \label{eq:fv-gam-hl}
    \norm{\mathcal{I}\parent{\Pi_{\hi,\lo,\lo}  (F,P_{>N^\gamma}v,\phi)}}_{Y_N^\nu}&\lesssim T^\frac{1}{2}N^{\nu-\sigma'-\gamma\nu+\delta}\norm{\japbrak{D_y}^{\sigma'}F}_{S_{N,D'}}\norm{\japbrak{D_y}^\nu v}_{X_{\leq N,D}}
    \,,\\
    \label{eq:fg-gam-hl}
    \norm{\mathcal{I}\parent{\Pi_{\hi,\lo,\lo}  (F,P_{> N^\gamma}G,\phi)}}_{Y_N^\nu}&\lesssim T^\frac{1}{2}N^{\nu-s-\gamma\sigma'}\norm{\japbrak{D_y}^sF}_{X_{N,D'}}\norm{\japbrak{D_y}^{\sigma'} G}_{S_{\leq N,D'}}
    \,, \\
    \label{eq:ff-hl}
    \norm{\mathcal{I}\parent{\Pi_{\hi,\lo,\lo}  (F,F,\phi)}}_{Y_N^\nu}&\lesssim T^\frac{1}{2}N^{\nu-s-\sigma'}\norm{\japbrak{D_y}^sF}_{X_{N,D'}}\norm{\japbrak{D_y}^{\sigma'} F}_{S_{ N,D'}}
    \,, \\
    \label{eq:wg-hl}
    \norm{\mathcal{I}\parent{\Pi_{\hi,\lo,\lo}  (w,G,\phi)}}_{Y_N^\nu}&\lesssim T^\frac{1}{2}\norm{\japbrak{D_y}^\nu w}_{X_{N,\alpha}\cap X_{\leq N,D}}\norm{\japbrak{D_y}^{\sigma'} G}_{S_{\leq N,D'}}
    \,, \\
    \label{eq:wv-hl}
    \norm{\mathcal{I}\parent{\Pi_{\hi,\lo,\lo}  (w,v,\phi)}}_{Y_N^\nu}&\lesssim T^\frac{1}{2}\norm{\japbrak{D_y}^\nu w}_{X_{N,\alpha}\cap X_{\leq N,D}}\norm{\japbrak{D_y}^\sigma v}_{S_{\leq N,D}}
    \,.
\end{align}
\end{prop}

Note that by homogeneity, the above inequalities are still true without the assumption  $\norm{\japbrak{D_y}^\sigma\phi}_{S_{\leq N,D}}\lesssim 1$, up to multiplying the upper bound by $\norm{\japbrak{D_y}^\sigma\phi}_{S_{\leq N,D}}$. We only make this assumption for shortness of notation.

\begin{prop}[Trilinear estimates with \emph{low-high-low} type interactions]
\label{prop:llh}
Let $\phi$ be of type $v$ or $G$, and assume that $\max(\nu-\sigma',\sigma)<\alpha<\nu$,  $D>\alpha$ and $2D+s+\sigma'<D'$. Assuming that $\norm{\japbrak{D_y}^\sigma\phi}_{S_{\leq N,D}}\lesssim 1$, we have
\begin{align} 
    \label{eq:wg}
    &\norm{\mathcal{I}\parent{\Pi_{\lo,\hi,\lo}  \parent{w,G,\phi}}}_{Y_N^\nu} \lesssim T^\frac{1}{2} \norm{\japbrak{D_y}^\nu w}_{X_{N,\alpha}}\norm{\japbrak{D_y}^{\sigma'}G}_{S_{\leq N, D'}}\,,\\
    \label{eq:wv}
    &\norm{\mathcal{I}\parent{\Pi_{\lo,\hi,\lo}  \parent{w,v,\phi}}}_{Y_N^\nu} \lesssim T^\frac{1}{2}\norm{\japbrak{D_y}^\sigma w}_{S_{N,\alpha}}\norm{\japbrak{D_y}^\nu v}_{ X_{\leq N,D}}\,,\\
    \label{eq:fg}
    &\norm{\mathcal{I}\parent{\Pi_{\lo,\hi,\lo}  \parent{F,G,\phi}}}_{Y_N^\nu} \lesssim T^\frac{1}{2} N^{\nu-\sigma'-s}\norm{\japbrak{D_y}^sF}_{X_{N,D'}}\norm{\japbrak{D_y}^{\sigma'}G}_{S_{\leq N, D'}} \,,\\
    \label{eq:fv}
    &\norm{\mathcal{I}\parent{\Pi_{\lo,\hi,\lo}  \parent{F,v,\phi}}}_{Y_N^\nu} 
    \lesssim T^\frac{1}{2}N^{-\sigma'+0}\norm{\japbrak{D_y}^{\sigma'} F}_{S_{N,D'}}\norm{\japbrak{D_y}^\nu v}_{ X_{\leq N,D}} \,.
\end{align}
\end{prop}

\begin{prop}[Trilinear estimates with \emph{low-high-high} type interactions]
\label{prop:lhh}
Let $\phi$ be of type $v$ or $G$, and assume that $\max(\nu-\sigma',\sigma)<\alpha<\nu$ and $2D+s+\sigma'<D'$. Assuming that $\norm{\japbrak{D_y}^\sigma\phi}_{S_{\leq N,D}}\lesssim 1$, we have
\begin{align} 
    \label{eq:wg-lhh}
    &\norm{\mathcal{I}\parent{\Pi_{\lo,\hi,\hi}  \parent{w,G,\phi}}}_{Y_N^\nu}
     \lesssim T^\frac{1}{2} \norm{\japbrak{D_y}^\nu w}_{X_{N,\alpha}}\norm{\japbrak{D_y}^{\sigma'}G}_{S_{\leq N, D'}}\,,\\
    \label{eq:wv-lhh}
    &\norm{\mathcal{I}\parent{\Pi_{\lo,\hi,\hi}  \parent{w,v,\phi}}}_{Y_N^\nu}
     \lesssim T^\frac{1}{2}\norm{\japbrak{D_y}^\sigma w}_{S_{N,\alpha}}\norm{\japbrak{D_y}^\nu v}_{X_{\leq N,D}}\,,\\
    \label{eq:fg-lhh}
    &\norm{\mathcal{I}\parent{\Pi_{\lo,\hi,\hi}  \parent{F,G,\phi}}}_{Y_N^\nu} \lesssim T^\frac{1}{2} N^{\nu-\sigma'-s+0}\norm{\japbrak{D_y}^sF}_{X_{N,D'}} \norm{\japbrak{D_y}^{\sigma'}G}_{S_{\leq N, D'}} \,,\\
     \label{eq:fv-lhh}
    &\norm{\mathcal{I}\parent{\Pi_{\lo,\hi,\hi}  \parent{F,v,\phi}}}_{Y_N^\nu} \lesssim T^\frac{1}{2}N^{-\sigma'+0}\norm{\japbrak{D_y}^{\sigma'} F}_{S_{N,D'}} \norm{\japbrak{D_y}^\nu v}_{X_{\leq N,D}} \,.
\end{align}
\end{prop}

\begin{prop}[Trilinear estimates with \emph{high-high} interactions type]
\label{prop:hh}
Let $\phi$ be of type $v$ or $G$, and assume that $\max(\nu-\sigma',\sigma)<\alpha<\nu$ and $2D+\nu+s+\sigma'<D'$. Assuming that $\norm{\japbrak{D_y}^\sigma\phi}_{S_{\leq N,D}}\lesssim 1$, we have
\begin{align} 
    \label{eq:wg-hh}
    &\norm{\mathcal{I}\parent{\Pi_{\hi,\hi}  \parent{w,G,\phi}}}_{Y_N^\nu}
    	\lesssim T^\frac{1}{2} \norm{\japbrak{D_y}^\nu w}_{X_{N,\alpha}}\norm{\japbrak{D_y}^{\sigma'}G}_{S_{\leq N, D'}}\,,\\
    \label{eq:wv-hh}
    &\norm{\mathcal{I}\parent{\Pi_{\hi,\hi}  \parent{w,v,\phi}}}_{Y_N^\nu}
    	 \lesssim T^\frac{1}{2}\norm{\japbrak{D_y}^\nu w}_{X_{N,\alpha}\cap X_{\leq N,D}} \norm{\japbrak{D_y}^\sigma v}_{S_{\leq N,D}}\,,\\
    \label{eq:fg-hh}
    &\norm{\mathcal{I}\parent{\Pi_{\hi,\hi}\parent{F,G,\phi}}}_{Y_N^\nu} 
    	\lesssim T^\frac{1}{2} N^{\nu-\sigma'-s+0}\norm{\japbrak{D_y}^sF}_{X_{N,D'}} \norm{\japbrak{D_y}^{\sigma'}G}_{S_{\leq N, D'}} \,,\\
    \label{eq:fv-hh}
    &\norm{\mathcal{I}\parent{\Pi_{\hi,\hi}\parent{F,v,\phi}}}_{Y_N^\nu}
    	 \lesssim T^\frac{1}{2}N^{-\sigma'+0}\norm{\japbrak{D_y}^{\sigma'} F}_{S_{N,D'}}
    	 \norm{\japbrak{D_y}^\nu v}_{X_{\leq N,D}} \,.
\end{align}
\end{prop}

The general strategy is as follows. The first term $\phi^{(1)}$ is of type $w$ or $F$. Hence, it is always localized at frequencies $\sim N$, in the sense that they belong to the  spaces $X_{N,\alpha}$ and $S_{N,\alpha}$.  When it is of type $F$, and the other terms are of type $w$, we place $F$ it in $L^\infty$ to exploit the improved stochastic bounds from Proposition~\ref{prop:srt-proba}.

\begin{proof}[Proof of Proposition~\ref{prop:hll}]
It is sufficient to put the third term in $S_{\leq N, D}$ with $\sigma$ derivatives, and to perform the same analysis as in~\cite{bringmann2021-ansatz} Section 7 to handle the first two terms. We only detail the first case.

~\\
$\bullet$\emph{Proof of~\eqref{eq:fv-gam-hl}}.
For fixed $M\geq1$, we have 
\begin{multline*}
 M^\nu\norm{P_M\parent{ \Pi_{\hi,\lo,\lo}\parent{F , P_{>N^\gamma}v, \phi }} }_{L_t^\frac{4}{3}L_x^1L_y^2} \\
    \lesssim  M^\nu\sum_{K\sim M}\sum_{N^\gamma < L_1\ll K}\sum_{L_2\ll K} \norm{P_M\parent{P_KF_N P_{L_1}w_{\leq \frac N2}P_{L_2}u_{\frac N2}}}_{L_t^\frac{4}{3}L_x^1L_y^2} \\
    \lesssim M^\nu T^\frac{1}{2}\sum_{K\sim M}\sum_{N^\gamma < L_1\ll M}\sum_{L_2\ll M}L_1^{-\nu}L_2^{-\sigma}K^{-\sigma'}\min\parentbig{\frac{N}{K},\frac{K}{N}}^{-D'} \\
    C_{N,D'}(K)\norm{\japbrak{D_y}^{\sigma'}P_KF}_{L_t^8L_x^4L_y^\infty}\norm{\japbrak{D_y}^\nu P_{L_1}v}_{L_t^\infty L_{x,y}^2}\norm{\japbrak{D_y}^\sigma P_{L_2}\phi}_{L_t^8L_x^4L_y^\infty}\,.   
\end{multline*}
Therefore, we obtain from the definitions of the norms and the assumption $D'\gg\max(\alpha + \sigma' + \gamma\nu,D + \nu)$, that
\begin{multline*}
    M^\nu\norm{P_M  \Pi_{\hi,\lo,\lo}F , P_{>N^\gamma}v, \phi }_{L_t^\frac{4}{3}L_x^1L_y^2} \\
    \lesssim T^\frac{1}{2}M^{\nu-\sigma'-\gamma\nu+0}\min\parentbig{\frac{N}{M},\frac{M}{N}}^{-D'}
    \norm{\japbrak{D_y}^{\sigma'}F}_{S_{N,D'}}\norm{\japbrak{D_y}^{\nu}v}_{X_{\leq N,D}}\norm{\japbrak{D_y}^{\sigma}\phi}_{S_{\leq N,D}}\,.
\end{multline*}

It remains to multiply by $C_{N,\alpha}(M)+ C_{\leq N, D}(M)$, and to sum over $M$. Since $D>\alpha$, we get
\begin{equation*}
    \begin{split}
            \sum_M &M^{\nu-\sigma'-\gamma\nu+0}\parent{C_{N,\alpha}(M)+C_{\leq N, D}(M)}\parentbig{1_{M\lesssim N} \parentbig{\frac{N}{M}}^{-D'} + 1_{M\gg N} \parentbig{\frac{M}{N}}^{-D'}} \\
            &=\sum_M M^{\nu-\sigma'-\gamma\nu+0}\parentbig{1_{M\lesssim N} \parentbig{\frac{N}{M}}^{-D'+\alpha} + 1_{M\gg N} \parentbig{\frac{M}{N}}^{-D'+D}} \\
            &=\sum_M N^{\nu-\sigma'-\gamma\nu+0}\parentbig{1_{M\lesssim N} \parentbig{\frac{N}{M}}^{-D'+\alpha-\nu+\sigma'+\gamma\nu} + 1_{M\gg N} \parentbig{\frac{M}{N}}^{-D'+D+\nu-\sigma'-\gamma\nu}} \\
            &\lesssim N^{\nu-\sigma'-\gamma\nu+0}\,,
    \end{split}
\end{equation*}

Similarly, in all the other \emph{high-low-low}-type frequency interactions estimated in Proposition ~\ref{prop:hll}, we put the third term in $L_t^8L_x^4L_y^\infty$, and we gain nothing from it. As for the quadratic interaction of the first two terms, the analysis is exactly as in~\cite{bringmann2021-ansatz}, and we do not write it here.
\end{proof}

\begin{proof}[Proof of Proposition~\ref{prop:llh}]
 For fixed $M\in2^N$, note that we have by definition 
\begin{multline}\label{ineq:decompo-llh}
    M^\nu\norm{P_M\Pi_{\lo,\hi,\lo}  \parent{\phi^{(1)},\phi^{(2)},\phi^{(3)}}}_{L_t^\frac{4}{3}L_x^1L_y^2} \\ \lesssim M^\nu\sum_{N_2\sim M}\sum_{N_1,N_3\ll N_2}\norm{P_{N_1}\phi^{(1)}P_{N_2}\phi^{(2)}P_{N_3}\phi^{(3)}}_{L_t^\frac{4}{3}L_x^1L_y^2}\,.
\end{multline}
~\\
$\bullet$\emph{Proof of~\eqref{eq:wg}}. In this case, $\phi^{(2)}=G$ and $\phi^{(1)}=w$. On the other hand, the second term $\phi^{(3)}:=\phi$, localized at frequencies $\lesssim N$, has no importance, and we always place it in $L^\infty$ with $\sigma$ derivatives in $y$. We do not gain anything from it. There holds
\begin{multline*}
     \sum_{N_2\sim M}\sum_{N_1,N_3\ll N_2}\norm{P_{N_1}w P_{N_2}G P_{N_3}\phi^{(3)}}_{L_t^\frac{4}{3}L_x^1L_y^2} 
    \\
    \lesssim T^\frac{1}{2}\norm{\japbrak{D_y}^\nu w}_{X_{N,\alpha}} 
    \parentbig{\sum_{N_1\ll M} \max\Big(\frac{N}{N_1},\frac{N_1}{N}\Big)^{-\alpha} N_1^{-\nu}}
    \\
    \sum_{N_2\sim M}N_2^{-\sigma'}\norm{\japbrak{D_y}^{\sigma'}P_{N_2}G}_{L_t^8L_x^4L_y^\infty} \norm{\japbrak{D_y}^\sigma\phi^{(3)}}_{S_{\leq N,D}}\\
    \lesssim T^\frac{1}{2}M^{-\sigma'}N^{-\alpha}\max\parentbig{1,\frac{M}{N}}^{-D'}\norm{\japbrak{D_y}^\nu w}_{X_{N,\alpha}} \norm{\japbrak{D_y}^{\sigma'}G}_{S_{\leq N,D'}} \norm{\japbrak{D_y}^\sigma\phi^{(3)}}_{S_{\leq N,D}}\,.
\end{multline*}
Note that we used the assumption $\alpha<\nu$ to obtain
\begin{equation}\label{sum:alpha-nu}
\sum_{N_1}\max\parentbig{\frac{N_1}{N},\frac{N}{N_1}}^{-\alpha}N_1^{-\nu } \lesssim N^{-\alpha}\,.
\end{equation}
Then, we multiply by $\parent{C_{\leq N, D}(M)+C_{N,\alpha}(M)}M^\nu$, and using that $\alpha>\nu-\sigma'$, $D>\alpha$ and $D'>D+\nu$, we sum over $M$ to conclude
\begin{multline*}
        \sum_M \parent{C_{N,\alpha}(M)+C_{\leq N,D}(M)} M^{\nu-\sigma'+0}N^{-\alpha}\max\parentbig{1,\frac{M}{N}}^{-D'} \\
    \lesssim \sum_{M\lesssim N} M^{\nu-\sigma'-\alpha+0} + N^{-\alpha}\sum_{M\gg N} M^{\nu-\sigma'+0}\parentbig{\frac{M}{N}}^{D-D'} \lesssim 1\,.
\end{multline*}
~\\
$\bullet$\emph{Proof of~\eqref{eq:fg}}.
We proceed similarly in this case, where $\phi^{(1)}=F$ and $\phi^{(2)}=G$. Once again, we use the localization of $F$ at frequencies $\sim N$, which is stronger than the localization of~$w$. There holds
\begin{multline*}
   \sum_{N_2\sim M}\sum_{N_1,N_3\ll N_2}\norm{ P_{N_1}F P_{N_2}G P_{N_3}\phi^{(3)}}_{L_t^\frac{4}{3}L_x^1L_y^2} \\
   \lesssim T^\frac{1}{2}\parentbig{\sum_{N_1\ll M} \max\parentbig{\frac{N}{N_1},\frac{N_1}{N}}^{-D'} N_1^{-s }}\norm{\japbrak{D_y}^s F}_{X_{N,D'}}\\
   \sum_{N_2\sim M}N_2^{-\sigma'+0}\norm{\japbrak{D_y}^{\sigma'}P_{N_2}G}_{L_t^8L_x^4L_y^\infty}\norm{\japbrak{D_y}^{\sigma}\phi^{(3)}}_{S_{\leq N,D}}\,.
\end{multline*}
Then, we observe that 
\begin{equation*}
    \begin{split}
   M^\nu\parentbig{\sum_{N_1\ll M} \max\parentbig{\frac{N}{N_1},\frac{N_1}{N}}^{-D'} & N_1^{-s }}\sum_{N_2\sim M}N_2^{-\sigma' +0}\max\parentbig{1,\frac{N_2}{N}}^{-D'}\\
    &\lesssim M^{\nu-\sigma' +0}\parentbig{\un_{M\lesssim N}\parentbig{\frac{N}{M}}^{-D'}M^{-s }+\un_{M\gg N}N^{-s }\parentbig{\frac{M}{N}}^{-D'}}\\
    &\lesssim N^{\nu-\sigma'-s}\parentbig{\un_{M\lesssim N}\parentbig{\frac{N}{M}}^{s-D'-\nu+\sigma'-0}+\un_{M\gg N}\parentbig{\frac{M}{N}}^{-D'+\nu-\sigma'+0}}\\
    &\lesssim N^{\nu-\sigma'-s}\max\parentbig{\frac{M}{N},\frac{N}{M}}^{-2D}\,,
    \end{split}
\end{equation*}
provided $D' > 2D + \nu + \sigma' + s$. Hence, we can multiply by $\parent{C_{\leq N, D}(M)+C_{N,\alpha}(M)}M^\nu$, sum over $M$ and obtain~\eqref{eq:fg}.

~\\
$\bullet$\emph{Proof of~\eqref{eq:wv}}. This case is easier since the derivatives fall onto the term $v$ which is in $\mathcal{H}^\nu$. We have $\phi^{(2)}=v$ and $\phi^{(1)}=w$. We place the high frequency term in $L^2$ in order to absorb $M^\sigma$, and we end up with 
\begin{equation*}
    \begin{split}
    \sum_{N_2\sim M}\sum_{N_1,N_3\ll N_2}  &\norm{P_{N_1}w P_{N_2}v P_{N_3}\phi^{(3)}}_{L_t^\frac{4}{3}L_x^1L_y^2} \\
        &\lesssim T^\frac{1}{2}\sum_{N_1,N_3\ll M}\sum_{N_2\sim M}\norm{P_{N_1}w}_{L_t^8L_x^4L_y^\infty}\norm{P_{N_3}\phi^{(3)}}_{L_t^8L_x^4L_y^\infty} \norm{P_{N_2}v}_{L_t^\infty L_{x,y}^2}\\
        &\lesssim T^\frac{1}{2}\parentbig{\sum_{N_1\ll M}\max\parentbig{\frac{N}{N_1},\frac{N_1}{N}}^{-\alpha}N_1^{-\sigma}}
        \sum_{N_2\sim M}\max\parentbig{1,\frac{N_2}{N}}^{-D}  N_2^{-\nu}\\
        &\hspace{150pt} \norm{\japbrak{D_y}^\sigma w}_{S_{N,\alpha}}\norm{\japbrak{D_y}^\sigma\phi^{(3)}}_{S_{\leq N,D}}\norm{\japbrak{D_y}^\nu P_{N_2}v}_{X_{\leq N,D}}\,.
    \end{split}
\end{equation*}
Since $\alpha>\sigma,$ we observe that the sum over $N_1$ is bounded by $N^{-\alpha}M^{\alpha-\sigma}$ when $M\lesssim N$ and $N^{-\sigma}$ when $M\gg N$, so that
\begin{multline*}
M^{\nu}\sum_{N_1\ll M}\max\parentbig{\frac{N}{N_1},\frac{N_1}{N}}^{-\alpha}N_1^{-\sigma} \sum_{N_2\sim M}\max\parentbig{1,\frac{N_2}{N}}^{-D}N_2^{-\nu}\\
	\lesssim \un_{M\lesssim N} N^{-\alpha}M^{\alpha-\sigma}+\un_{M\gg N}\parentbig{\frac{M}{N}}^{-D} N^{-\sigma}\\
	\lesssim \un_{M\lesssim N} \parentbig{\frac{N}{M}}^{-\alpha}M^{-\sigma}+\un_{M\gg N}\parentbig{\frac{M}{N}}^{-D}\,.
\end{multline*}
We deduce~\eqref{eq:wv} after multiplying by $\parent{C_{\leq N, D}(M)+C_{N,\alpha}(M)}$ and summing over $M$, making the norm $\norm{\japbrak{D_y}^\nu v}_{X_{\leq N,D}}$ appear. 

~\\
$\bullet$\emph{Proof of~\eqref{eq:fv}.}
This proof is identical than the one above. Observe that we only used the localization of $w$ at frequency $N$ in the Strichartz space $S_{N,\alpha}$, and the localization of $F$ is even better. Indeed, we have
\begin{multline*}
\sum_{N_2\sim M}\sum_{N_1,N_3\ll N_2}  \norm{P_{N_1}F P_{N_2}v P_{N_3}\phi^{(3)}}_{L_t^\frac{4}{3}L_x^1L_y^2} \\
        \lesssim T^\frac{1}{2}\parentbig{\sum_{N_1\ll M}\max\parentbig{\frac{N}{N_1},\frac{N_1}{N}}^{-D'}N_1^{-\sigma'}}
        \sum_{N_2\sim M}\max\parentbig{1,\frac{N_2}{N}}^{-D}  N_2^{-\nu}\\
	 \norm{\japbrak{D_y}^{\sigma'} F}_{S_{N,D'}}\norm{\japbrak{D_y}^\sigma\phi^{(3)}}_{S_{\leq N,D}}\norm{\japbrak{D_y}^\nu P_{N_2}v}_{X_{\leq N,D}}\,.
\end{multline*}
Since $D'>\sigma',$ we have
\[
\un_{M\gg N} \sum_{N_1\ll N}\parentbig{\frac{N}{N_1}}^{-D'}N_1^{-\sigma'}\lesssim N^{-\sigma'}\,,
\]
\[
\un_{M\gg N} \sum_{N\lesssim N_1\ll M} \parentbig{\frac{N_1}{N}}^{-D'}N_1^{-\sigma'}\lesssim N^{-\sigma'}\,,
\]
\[
\un_{M\lesssim N} \sum_{N_1\ll M}\parentbig{\frac{N}{N_1}}^{-D'}N_1^{-\sigma'}\lesssim N^{-D'}M^{D'-\sigma'}\,,
\]
so that
\begin{multline*}
M^{\nu}\parentbig{\sum_{N_1\ll M}\max\parentbig{\frac{N}{N_1},\frac{N_1}{N}}^{-D'}N_1^{-\sigma'}} \sum_{N_2\sim M}\norm{\japbrak{D_y}^\nu P_{N_2}v}_{X_{\leq N,D}} N_2^{-\nu}\\
	\lesssim \un_{M\lesssim N} \parentbig{\frac{N}{M}}^{-D'+\sigma'-0}N^{-\sigma'} + \un_{M\gg N} \parentbig{\frac{M}{N}}^{-D}N^{-\sigma'}\,.
\end{multline*}
Since $D'>\sigma'+\alpha$, we deduce~\eqref{eq:fv} after multiplying by $\parent{C_{\leq N, D}(M)+C_{N,\alpha}(M)}$ and summing over $M$, making the norm $\norm{\japbrak{D_y}^\nu v}_{X_{\leq N,D}}$ appear.
\end{proof}

\begin{proof}[Proof of Proposition~\ref{prop:lhh}]
In this situation, inequality~\eqref{ineq:decompo-llh} becomes restricted to the indices $N_2\sim N_3\gtrsim M$ such that $N_1\ll N_2$:
\begin{multline*}
    M^\nu\norm{P_M\Pi_{\lo,\hi,\hi}  \parent{\phi^{(1)},\phi^{(2)},\phi^{(3)}}}_{L_t^\frac{4}{3}L_x^1L_y^2} 
    \lesssim M^\nu\sum_{N_2\sim N_3\gtrsim M}\sum_{N_1\ll N_2}\norm{P_{N_1}\phi^{(1)}P_{N_2}\phi^{(2)}P_{N_3}\phi^{(3)}}_{L_t^\frac{4}{3}L_x^1L_y^2}\,.
\end{multline*}
~\\
$\bullet$\emph{Proof of~\eqref{eq:wg-lhh}.}
Now, we observe that 
\begin{multline*}
\sum_{N_2\sim N_3\gtrsim M}\sum_{N_1\ll N_2}\norm{P_{N_1}w P_{N_2}G P_{N_3}\phi^{(3)}}_{L_t^\frac{4}{3}L_x^1L_y^2}\\
	\lesssim T^{\frac 12}\sum_{ N_2\gtrsim M}\sum_{N_1\ll N_2}\max\parentbig{\frac{N}{N_1},\frac{N_1}{N}}^{-\alpha} 
	\max\parentbig{1,\frac{N_2}{N}}^{-D'} N_1^{-\nu} N_2^{-\sigma'}\\
	\norm{\japbrak{D_y}^{\nu}w}_{X_{N,\alpha}}\norm{\japbrak{D_y}^{\sigma}\phi^{(3)}}_{S_{\leq N,D}}\norm{\japbrak{D_y}^{\sigma'}G}_{S_{\leq N,D'}} \,.
\end{multline*}
Then, according to~\eqref{sum:alpha-nu}, the sum over $N_1$ is bounded by $N^{-\alpha}$. Moreover, one can check that
\begin{equation}\label{sum:N3}
\sum_{M\lesssim N_2}\max\left(1,\frac{N_2}{N}\right)^{-D'}N_2^{-\sigma'}
	\lesssim
\begin{cases}
M^{-\sigma'}& \text{ if } M\lesssim N\,,\\
 \parent{\frac{M}{N}}^{-D'} M^{-\sigma'} & \text{ if } M\gg M\,,
\end{cases}
\end{equation}
therefore we have
\begin{multline*}
M^{\nu}\sum_{ N_2\gtrsim M}\sum_{N_1\ll N_2}\max\parentbig{\frac{N}{N_1},\frac{N_1}{N}}^{-\alpha} 
	\max\parentbig{1,\frac{N_2}{N}}^{-D'} N_1^{-\nu} N_2^{-\sigma'}\\
		\lesssim  \un_{M\lesssim N}N^{-\alpha}M^{\nu-\sigma'}+\un_{M\gg N}\parentbig{\frac{M}{N}}^{-D'}N^{-\alpha}M^{\nu-\sigma'}\\
		\lesssim \un_{M\lesssim N}\parentbig{\frac{N}{M}}^{-\alpha}M^{\nu-\sigma'-\alpha}+\un_{M\gg N}\parentbig{\frac{M}{N}}^{-D'+\alpha}M^{\nu-\sigma'-\alpha} \,.
\end{multline*}
Estimate~\eqref{eq:wg-hh} comes from multyplying by $\parent{C_{\leq N, D}(M)+C_{N,\alpha}(M)}$ and summing over $M$, using that $\alpha>\nu-\sigma'$.

~\\
$\bullet$\emph{Proof of~\eqref{eq:wv-lhh}.}
Similarly, we write
\begin{multline*}
\sum_{N_2\sim N_3\gtrsim M}\sum_{N_1\ll N_2}\norm{P_{N_1}w P_{N_2}v P_{N_3}\phi^{(3)}}_{L_t^\frac{4}{3}L_x^1L_y^2}\\
	\lesssim T^{\frac 12}\sum_{ N_2\gtrsim M}\sum_{N_1\ll N_2}
	\max\parentbig{\frac{N}{N_1},\frac{N_1}{N}}^{-\alpha} \max\parentbig{1,\frac{N_2}{N}}^{-D} N_1^{-\sigma} N_2^{-\nu}\\
	\norm{\japbrak{D_y}^{\sigma}w}_{S_{N,\alpha}}\norm{\japbrak{D_y}^{\sigma}\phi^{(3)}}_{S_{\leq N,D}}\norm{\japbrak{D_y}^{\nu}P_{N_2}v}_{X_{\leq N,D}} \,.
\end{multline*}
Since $\alpha>\sigma$, we observe that the sum over $N_1$ is bounded by $\parent{\frac{N}{M}}^{-\alpha}M^{-\sigma}$ when $M\lesssim N$ and $N^{-\sigma}$ when $M\gg N$. Moreover, the sum over $N_2$ is bounded as in~\eqref{sum:N3} by replacing $D'$ by $D$ and $\sigma'-0$ by $\nu$, which leads to 
\begin{multline*}
 M^{\nu}\sum_{ N_2\gtrsim M}\sum_{N_1\ll N_2}\max\parentbig{\frac{N}{N_1},\frac{N_1}{N}}^{-\alpha}
	\max\parentbig{1,\frac{N_2}{N}}^{-D} N_1^{-\sigma} N_2^{-\nu}\\
	\lesssim  \un_{M\lesssim N} \parentbig{\frac{N}{M}}^{-\alpha}M^{-\sigma}+ \un_{M\gg N}\parentbig{\frac{M}{N}}^{-D} N^{-\sigma}\,.
\end{multline*}
It remains to multiply by $\parent{C_{\leq N, D}(M)+C_{N,\alpha}(M)}$ and sum over $M$ to get the norm $\norm{\japbrak{D_y}^{\nu}v}_{X_{\leq N,D}}$.

~\\
$\bullet$\emph{Proof of~\eqref{eq:fg-lhh}.}
We adopt the same strategy but we put the first term in the adapted space $X_{N,D'}$ after taking $s$ derivatives:
\begin{multline*}
\sum_{N_2\sim N_3\gtrsim M}\sum_{N_1\lesssim N_2}\norm{P_{N_1}F P_{N_2}G P_{N_3}\phi^{(3)}}_{L_t^\frac{4}{3}L_x^1L_y^2}\\
	\lesssim T^{\frac 12}\sum_{ N_2\gtrsim M}\sum_{N_1\lesssim N_2}\max\parentbig{\frac{N}{N_1},\frac{N_1}{N}}^{-D'} 
	\max\parentbig{1,\frac{N_2}{N}}^{-D'} N_1^{-s} N_2^{-\sigma'}\\
	\norm{\japbrak{D_y}^{s}F}_{X_{N,D'}}\norm{\japbrak{D_y}^{\sigma}\phi^{(3)}}_{S_{\leq N,D}}\norm{\japbrak{D_y}^{\sigma'}G}_{S_{\leq N,D'}} \,.
\end{multline*}
Since $D'>s$, the sum over $N_1$ is bounded by $\parent{\frac{N}{M}}^{-D'}M^{-s}$ when $M\lesssim N$ and $N^{-s}$ when $M\gg N$. Moreover, the sum over $N_2$ is bounded as in~\eqref{sum:N3} and we get
\begin{multline*}
 M^{\nu}\sum_{ N_2\gtrsim M}\sum_{N_1\lesssim N_2} \max\parentbig{\frac{N}{N_1},\frac{N_1}{N}}^{-D'}
	\max\parentbig{1,\frac{N_2}{N}}^{-D'} N_1^{-s} N_2^{-\sigma'}\\
	\lesssim \un_{M\lesssim N}\parentbig{\frac{N}{M}}^{-D'}M^{\nu-s-\sigma'}
	+\un_{M\gg N}\parentbig{\frac{M}{N}}^{-D'} N^{-s}M^{\nu-\sigma'}\\
	\lesssim \un_{M\lesssim N}\parentbig{\frac{N}{M}}^{-D'-\nu+s+\sigma'-0}N^{\nu-s-\sigma'}
	+\un_{M\gg N}\parentbig{\frac{M}{N}}^{-D'-\nu+\sigma'} N^{\nu-s-\sigma'}\,.
\end{multline*}
Since $D'>\alpha+s+\sigma'$ and $D'>D+\sigma'$, multiplying by $\parent{C_{\leq N, D}(M)+C_{N,\alpha}(M)}$ the sum over $M$ is bounded by $N^{\nu-s-\sigma'}$.

~\\
$\bullet$\emph{Proof of~\eqref{eq:fv-lhh}.}
Finally, we put the term $F$ in $S_{N,D'}$:
\begin{multline*}
\sum_{N_2\sim N_3\gtrsim M}\sum_{N_1\lesssim N_2}\norm{P_{N_1}F P_{N_2}v P_{N_3}\phi^{(3)}}_{L_t^\frac{4}{3}L_x^1L_y^2}\\
	\lesssim T^{\frac 12}\sum_{N_2\gtrsim M} \sum_{N_1\lesssim N_2} \max\parentbig{\frac{N}{N_1},\frac{N_1}{N}}^{-D'} \max\parentbig{1,\frac{N_2}{N}}^{-D} N_1^{-\sigma'} N_2^{-\nu}\\
	\norm{\japbrak{D_y}^{\sigma'}F}_{S_{N,D'}}\norm{\japbrak{D_y}^{\sigma}\phi^{(3)}}_{S_{\leq N,D}}\norm{\japbrak{D_y}^{\nu}v}_{X_{\leq N,D}} \,.
\end{multline*}
The sum over $N_1$ is bounded as above by $\parent{\frac{N}{M}}^{-D'}M^{-\sigma'}$ when $M\lesssim N$ and $N^{-\sigma'}$ when $M\gg N$, and the sum over $N_2$ is bounded as in~\eqref{sum:N3} by $M^{-\nu}$ if $M\lesssim N$ and  $\parent{\frac{M}{N}}^{-D}M^{-\nu}$ if $M\gg N$. We get
\begin{multline*}
 M^{\nu}\sum_{N_2\gtrsim M}\sum_{N_1\lesssim N_2} \max\parentbig{\frac{N}{N_1},\frac{N_1}{N}}^{-D'}
	\max\parentbig{1,\frac{N_2}{N}}^{-D} N_1^{-\sigma'} N_2^{-\nu}\\
	\lesssim \un_{M\lesssim N}\parentbig{\frac{N}{M}}^{-D'}M^{-\sigma'}
	+\un_{M\gg N}\parentbig{\frac{M}{N}}^{-D} N^{-\sigma'}\,.
\end{multline*}
Since $D'>\alpha+\sigma'$, after multiplying by $\parent{C_{\leq N, D}(M)+C_{N,\alpha}(M)}$ the sum over $M$ is bounded by $N^{-\sigma'}$.
\end{proof}

\begin{proof}[Proof of Proposition~\ref{prop:hh}]
In this situation, inequality~\eqref{ineq:decompo-llh} becomes restricted to the indices $N_1\sim N_2\gtrsim M$ such that $N_3\lesssim N_1$:
\begin{multline*}
    M^\nu\norm{P_M\Pi_{\hi,\hi}  \parent{\phi^{(1)},\phi^{(2)},\phi^{(3)}}}_{L_t^\frac{4}{3}L_x^1L_y^2} 
    \lesssim M^\nu\sum_{N_1\sim N_2\gtrsim M}\sum_{N_3\lesssim N_1}\norm{P_{N_1}\phi^{(1)}P_{N_2}\phi^{(2)}P_{N_3}\phi^{(3)}}_{L_t^\frac{4}{3}L_x^1L_y^2}\,.
\end{multline*}
~\\
$\bullet$\emph{Proof of~\eqref{eq:wg-hh}.}
Now, we observe that 
\begin{multline*}
\sum_{N_1\sim N_2\gtrsim M}\sum_{N_3\lesssim N_1}\norm{P_{N_1}w P_{N_2}G P_{N_3}\phi^{(3)}}_{L_t^\frac{4}{3}L_x^1L_y^2}\\
	\lesssim T^{\frac 12}\sum_{N_1\sim N_2\gtrsim M}\max\parentbig{\frac{N}{N_1},\frac{N_1}{N}}^{-\alpha} 
	\max\parentbig{1,\frac{N_2}{N}}^{-D'} N_1^{-\nu} N_2^{-\sigma'}\\
	\norm{\japbrak{D_y}^{\nu}w}_{X_{N,\alpha}}\norm{\japbrak{D_y}^{\sigma}\phi^{(3)}}_{S_{\leq N,D}}\norm{\japbrak{D_y}^{\sigma'}G}_{S_{\leq N,D'}} \,.
\end{multline*}
Then, since $N_2\sim N_1$ and $\alpha<\nu$, we observe that 
\[\sum_{M\lesssim N_1\ll N}\parentbig{\frac{N}{N_1}}^{-\alpha}N_1^{-\nu-\sigma'}\lesssim N^{-\alpha}M^{\alpha-\nu-\sigma'}=\parentbig{\frac NM}^{-\alpha} M^{-\nu-\sigma'}\,,
\] 
\[\sum_{M\lesssim N\lesssim N_1}\parentbig{\frac{N_1}{N}}^{-\alpha-D'}N_1^{-\nu-\sigma'}
	\lesssim N^{-\nu-\sigma'}=\parentbig{\frac NM}^{-\nu}M^{-\nu}N^{-\sigma'}\,,
\]
\[
\sum_{N_1\gtrsim M\gg N}\parentbig{\frac{N_1}{N}}^{-\alpha-D'}N_1^{-\nu-\sigma'}
	\lesssim \parentbig{\frac{M}{N}}^{-\alpha-D'} M^{-\nu-\sigma'}\,,
\]
therefore since $D'>D$, we have
\begin{multline*}
M^{\nu}\sum_{N_1\gtrsim M}\max\parentbig{\frac{N}{N_1},\frac{N_1}{N}}^{-\alpha} 
	\max\parentbig{1,\frac{N_1}{N}}^{-D'} N_1^{-\nu-\sigma'}\\
		\lesssim \parentbig{\un_{M\lesssim N}\parentbig{\frac{N}{M}}^{-\alpha}+\un_{M\gg N}\parentbig{\frac{M}{N}}^{-D'}}M^{-\sigma'}\,.
\end{multline*}
Estimate~\eqref{eq:wg-hh} comes from multyplying by $\parent{C_{\leq N, D}(M)+C_{N,\alpha}(M)}$ and summing over $M$.
~\\
$\bullet$\emph{Proof of~\eqref{eq:wv-hh}.}
The proof is identical as the proof of~\eqref{eq:wg-hh}, indeed we only need to replace $G$ by $v$, $D'$ by $D$ dans $\sigma'$ by $\sigma$ in the argument.
~\\
$\bullet$\emph{Proof of~\eqref{eq:fg-hh}.}
We adopt the same strategy but we put the first term in the adapted space $X_{N,D'}$ after taking $s$ derivatives:
\begin{multline*}
\sum_{N_1\sim N_2\gtrsim M}\sum_{N_3\lesssim N_1}\norm{P_{N_1}F P_{N_2}G P_{N_3}\phi^{(3)}}_{L_t^\frac{4}{3}L_x^1L_y^2}\\
	\lesssim T^{\frac 12}\sum_{N_1\sim N_2\gtrsim M}\max\parentbig{\frac{N}{N_1},\frac{N_1}{N}}^{-D'} 
	\max\parentbig{1,\frac{N_2}{N}}^{-D'} N_1^{-s} N_2^{-\sigma'}\\
	\norm{\japbrak{D_y}^{s}F}_{X_{N,D'}}\norm{\japbrak{D_y}^{\sigma}\phi^{(3)}}_{S_{\leq N,D}}\norm{\japbrak{D_y}^{\sigma'}G}_{S_{\leq N,D'}} \,.
\end{multline*}
Then we have
\[
\un_{M\lesssim N}\sum_{M\lesssim N_1\sim N_2\lesssim N}\parentbig{\frac{N}{N_1}}^{-D'} N_1^{-s-\sigma'}
	\lesssim N^{-s-\sigma'}\,,
\]
\[
\un_{M\lesssim N}\sum_{N\lesssim N_1\sim N_2}\parentbig{\frac{N_1}{N}}^{-2D'} N_1^{-s-\sigma'}
	\lesssim N^{-s-\sigma'}\,,
\]
\[
\un_{M\gg N}\sum_{M\lesssim N_1\sim N_2}\parentbig{\frac{N_1}{N}}^{-2D'} N_1^{-s-\sigma'}
	\lesssim \parentbig{\frac{M}{N}}^{-2D'}M^{-s-\sigma'}\,,
\]
so that
\begin{multline*}
 M^{\nu}\sum_{N_1\sim N_2\gtrsim M}\max\parentbig{\frac{N}{N_1},\frac{N_1}{N}}^{-D'}
	\max\parentbig{1,\frac{N_2}{N}}^{-D'} N_1^{-s} N_2^{-\sigma'}\\
	\lesssim \un_{M\lesssim N}\parentbig{\frac{N}{M}}^{-\alpha}M^{\nu-\alpha}N^{\alpha-s-\sigma'}
	+\un_{M\gg N}\parentbig{\frac{M}{N}}^{-2D'+\nu-s-\sigma'} N^{\nu-s-\sigma'}\,.
\end{multline*}
After multiplying by $\parent{C_{\leq N, D}(M)+C_{N,\alpha}(M)}$ the sum over $M$ is bounded by $N^{\nu-s-\sigma'}$ since $2D'>D+\nu-s-\sigma'$ by assumption.

~\\
$\bullet$\emph{Proof of~\eqref{eq:fv-hh}.}
Finally, we put the term $F$ in $S_{N,D'}$:
\begin{multline*}
\sum_{N_1\sim N_2\gtrsim M}\sum_{N_3\lesssim N_1}\norm{P_{N_1}F P_{N_2}v P_{N_3}\phi^{(3)}}_{L_t^\frac{4}{3}L_x^1L_y^2}\\
	\lesssim T^{\frac 12}\sum_{N_1\sim N_2\gtrsim M}\max\parentbig{\frac{N}{N_1},\frac{N_1}{N}}^{-D'} \max\parentbig{1,\frac{N_2}{N}}^{-D} N_1^{-\sigma'} N_2^{-\nu}\\
	\norm{\japbrak{D_y}^{\sigma'}F}_{S_{N,D'}}\norm{\japbrak{D_y}^{\sigma}\phi^{(3)}}_{S_{\leq N,D}}\norm{\japbrak{D_y}^{\nu}v}_{X_{\leq N,D}} \,.
\end{multline*}
Then we have
\[
\un_{M\lesssim N}\sum_{M\lesssim N_1\sim N_2\lesssim N}\parentbig{\frac{N}{N_1}}^{-D'} N_1^{-\sigma'-\nu}
	\lesssim N^{-\sigma'-\nu}\,,
\]
\[
\un_{M\lesssim N}\sum_{N\lesssim N_1\sim N_2}\parentbig{\frac{N_1}{N}}^{-D'-D} N_1^{-\sigma'-\nu}
	\lesssim N^{-\sigma'-\nu}\,,
\]
\[
\un_{M\gg N}\sum_{M\lesssim N_1\sim N_2}\parentbig{\frac{N_1}{N}}^{-D'-D} N_1^{-\sigma'-\nu}
	\lesssim \parentbig{\frac{M}{N}}^{-D'-D}M^{-\sigma'-\nu}\,,
\]
so that
\begin{multline*}
 M^{\nu}\sum_{N_1\sim N_2\gtrsim M}\max\parentbig{\frac{N}{N_1},\frac{N_1}{N}}^{-D'}
	\max\parentbig{1,\frac{N_2}{N}}^{-D} N_1^{-\sigma'} N_2^{-\nu}\\
	\lesssim \un_{M\lesssim N}\parentbig{\frac{N}{M}}^{-\alpha}M^{\nu-\alpha}N^{\alpha-\nu-\sigma'}
	+\un_{M\gg N}\parentbig{\frac{M}{N}}^{-D'-D-\sigma'} N^{\sigma'}\,.
\end{multline*}
After multiplying by $\parent{C_{\leq N, D}(M)+C_{N,\alpha}(M)}$ the sum over $M$ is bounded by  $N^{-\sigma'}$.

\end{proof}




\appendix
\section{Remarks on ill-posedness}
\label{sec:ill-posedness}

\subsection{Semilinear ill-posedness}
From the traveling wave profiles for the Szeg\H{o} equation, we cook up a one-parameter family of profiles from which we deduce that the bilinear estimate only holds when $s\geq\frac{1}{2}$. Then, we proceed as in Remark 2.12 from~\cite{bgt-05} to deduce that the flow map cannot be of class $\mathcal{C}^3$ at the origin when $\frac{1}{4}<s<\frac{1}{2}$. We recall that $\mathcal{H}^s=H^{2s}_xL^2_y\cap L^2_xH^{s}_y$.


\begin{theorem}[Semilinear ill-posedness]
If there exists a local in time flow map on $\mathcal{H}^s$ with regularity $\mathcal{C}^3$ at the origin, then $s\geq\frac 12$.
\end{theorem}

\begin{proof} 
As a corollary of~\cite{bgt-05}, Remark 2.12, if there exists a $\mathcal{C}^3$ local in time flow map at the vicinity of the origin in the space $\mathcal{H}^s$, then the following bilinear estimate holds:
\begin{equation}\label{eq:Strichartz_disp}
\|e^{it\A}f\|_{L^4([0,1]\times \R^2)}\lesssim \|f\|_{\dot{\mathcal{H}}^{\frac s2}}\,.
\end{equation}
As usual, we use a one-parameter family of stationary solutions to evidence the instabilities of the flow-map. Namely, we consider a Gaussian distribution $G$, a family of traveling waves profiles for the Szeg\H{o} equation $K_{\rho}(y)$, for $\rho\in\intervaloo{0}{+\infty}$, and set
\[
f(x,y)=G(x)K_{\rho}(y)\,,\quad K_\rho(y)=\frac{1}{y+i\rho}\,.
\]
First, let us look at the scaling of $L_{t,x}^4$-norm of $f_\rho$ as $\rho$ goes to 0. Since $K_{\rho}$ is a traveling wave for the Szeg\H{o} equation on the line~\cite{Pocovnicu2011}, and in particular a traveling wave for equation~\eqref{eq:NLSHW}, one can see that for every $t$, there holds
\[
\|e^{it|D_y|}K_{\rho}\|_{L^4(\R_y)}=\|K_{\rho}\|_{L^4(\R_y)}
\]
with
\[
\|K_{\rho}\|_{L^4(\R_y)}^4=\int_{\R}\frac{1}{(y^2+\rho^2)^2}\dd y=\rho^{-3}\int_{\R}\frac{1}{(v^2+1)^2}\dd v\,.
\]
This implies that as $\rho\to 0$,
\[
\|e^{it\A} f\|_{L^4([0,1]\times \R^2_{x,y})}
	=\|e^{it\partial_{xx}}G\|_{L^4([0,1]\times \R_x)}\|K_{\rho}\|_{L^4(\R_y)}
	\sim  C\rho^{-\frac 34}\,.
\]

Next, let us look at the $L_x^2\dot{H}_y^{\frac{s}{2}}$ component of the $\dot{\mathcal{H}}^s$-norm of $f_\rho$. It reads
\[
\|f\|_{L^2_x\dot{H}_y^\frac{s}{2}}=\|G\|_{L^2_x}\|K_{\rho}\|_{H_y^\frac{s}{2}}\,,
\]
and we compute
\[
\|K_{\rho}\|_{\dot{H}_y^\frac{s}{2}}^2
	=\int_{\R}\int_{\R}\frac{|K_{\rho}(y+h)-K_{\rho}(y)|^2}{|h|^{1+s}}\dd h\dd y\,.
\]
Since 
\[K_{\rho}(y+h)-K_{\rho}(y)=\frac{-h}{(y+h+i\rho)(y+i\rho)}\,,\]
we deduce that
\[
\|K_{\rho}\|_{\dot{H}^{\frac{s}{2}}_y}^2
	=\int_{\R}\int_{\R}|h|^{1-\frac s2}\frac{1}{((y+h)^2+\rho^2)(y^2+\rho^2)}\dd h\dd y\,.
\]
We make the change of variable $y=\rho u$ and $h=\rho v$ to get
\[
\|K_{\rho}\|_{\dot{H}^{\frac{s}{2}}_y}^2
	= \rho^{-1-s}\int_{\R}\int_{\R}|v|^{1-s}\frac{1}{((u+v)^2+1)(u^2+1)}\dd u\dd v\sim C'\rho^{-1-s}\,,
\]
Moreover, 
\[
\norm{f}_{\dot{H}_x^sL_y^2} = \norm{\e^{it\partial_{xx}}G}_{\dot{H}_x^s}\norm{K_\rho}_{L_y^2} = \rho^{-\frac{1}{2}}\parent{\int\frac{1}{1+y^2}\dd y}^\frac{1}{2}\sim C''\rho^{-\frac 12}\,,
\]
so that 
\[
\|f\|_{\dot{\mathcal{H}}^s} \underset{\rho\to0}{\sim}C'\rho^{-\frac{1}{2}-\frac{s}{2}}\,.
\]
Since inequality~\eqref{eq:Strichartz_disp} has to hold as $\rho\to 0$, we see that necessarily, $\rho^{-\frac 34}\lesssim \rho^{-\frac 12-s}$, therefore
\[
\frac{3}{4}\leq\frac{1}{2}+\frac{s}{2}\,,
\]
leading to the result $\frac{1}{2}\leq s$.
\end{proof}

\subsection{Pathological set}

An adaptation of the arguments from~\cite{BurqGerardTzvetkov2005multilinear} implemented in~\cite{Kato2021} implies the following ill-posedness result.
\begin{theorem}[Norm inflation~\cite{Kato2021}] Let $s<\frac{1}{4}$.
For every bounded set $B$ of $\mathcal{H}^s$ and for every $T>0$, the flow map cannot be extended as a continuous map from $B$ to $\mathcal{C}([-T,T],\mathcal{H}^s)$. More precisely, there exists a sequence $(t_n)_{n\in\mathbb{N}}$ of positive numbers tending to zero and a sequence $(u_N(t))_{n\gg 1}$ of $\mathcal{C}^{\infty}(\mathbb{R}^2)$ solutions of~\eqref{eq:NLSHW} defined for $t\in[0,t_n]$, and as $n\to\infty$
\[
\|u_n(0)\|_{\mathcal{H}^s}\to 0\,,
\]
\[
\|u_n(t_n)\|_{\mathcal{H}^s}\to\infty\,.
\]
\end{theorem}

Non-uniform continuity of the flow map for $s=\frac{1}{4}$. has also been investigated in~\cite{Kato2021}. Using the method of Sun and Tzvetkov~\cite{SunTzvetkov2020pathological} adapted to Schrödinger-type equations in~\cite{CampsGassot2022}, we make precise this norm inflation result in the context of probabilistic well-posedness. We fix $\rho\in\mathcal{C}_c^{\infty}(\R^2)$, valued in $[0,1]$, such that $\int_{\R^2}\rho(x)\dd x=1$ and $\rho$ vanishes for $x^2+y^2\geq \frac{1}{10^4}$. Due to the anisotropy, we define an approximate identity $(\rho_{\varepsilon})_{\varepsilon>0}$ of the form
\[
\rho_{\varepsilon}(x,y):=\frac{1}{\varepsilon^3}\rho\left(\frac{x}{\varepsilon},\frac{y}{\varepsilon^2}\right)\,.
\]

\begin{theorem}[Generic ill-posedness for~\eqref{eq:NLSHW}]
Let $s<1/2$. There exists a dense set $S\subset X^s$ such that for every $f\in S$, the family of local solutions $(u^{\varepsilon})_{\varepsilon>0}$ of~\eqref{eq:NLSHW}  with initial data $\rho_{\varepsilon}\ast f$ does not converge. More precisely, there exist $\varepsilon_n\to 0$ and $t_n\to 0$ such that $u^{\varepsilon_n}(t_n)$ is well-defined and
\[
\lim_{n\to\infty}\|u^{\varepsilon_n}(t_n)\|_{\mathcal{H}^s}=\infty\,.
\]
\end{theorem}

Since the argument is the same as in~\cite{CampsGassot2022},  we only make a sketch of the proof adapted to the anisotropy. We construct unstable profiles with growing $\mathcal{H}^s$ norm based on the bounded time-periodic solution $V(t)=e^{it}$ to the ODE
\[
\begin{cases}
iV'+|V|^2V=0\,,
\\
V(0)=1\,.
\end{cases}
\]
Let $\varphi\in\mathcal{C}_{c}^{\infty}(\{x^2+y^2\leq 1\})$, radial such that $0\leq \varphi\leq 1$. Let $\gamma$ to be chosen later, and
\[
\kappa_n:=(\log(n))^{-\gamma}
\,,\quad \lambda_n:=\kappa_n n^{\frac{3}{2}-2s}\,.
\]
We define
\begin{equation*}
v_n(0,x,y):=\lambda_n\varphi(nx,n^2y)\,.
\end{equation*}
We recall that
\[
\rho_{\varepsilon}(x,y):=\frac{1}{\varepsilon^3}\rho\left(\frac{x}{\varepsilon},\frac{y}{\varepsilon^2}\right)\,,
\]
and define
\[
v_n^{\varepsilon}(0):=\rho_{\varepsilon}\ast v_n(0)\,.
\]
The profile
\begin{equation*}
v_n^{\varepsilon}(t,x,y):=v_n^{\varepsilon}(0,x,y) V(t|v_n^{\varepsilon}(0,x,y)|^2)
\end{equation*}
is solution with initial data $v_n^{\varepsilon}(0)$ to
\[
i\partial_t v_n^{\varepsilon}+|v_n^{\varepsilon}|^2v_n^{\varepsilon}=0\,.
\]

We now fix the parameters
\[
\varepsilon_n:=\frac{1}{100n}\,,\quad
t_n:=\log(n)^{2\beta}n^{2(2s-3/2)}=\lambda_n^{-2} \log(n)^{2(\beta-\gamma)}\,, \quad 0<\gamma<\beta<1\,.\]

Lemma~2.1 in~\cite{CampsGassot2022} becomes as follows.
\begin{lem}[Growth of the ODE profile]
Let $0\leq s<\frac 12$.
\begin{enumerate}
\item For every $\varepsilon\leq\varepsilon_n$, there holds 
\[\|v_n^{\varepsilon}(t_n)\|_{H^{2s}_xL^2_y}\gtrsim \kappa_n(\lambda_n^2t_n)^{2s}\,,\quad \|v_n^{\varepsilon}(t_n)\|_{L^2_xH^{s}_y}\gtrsim \kappa_n(\lambda_n^2t_n)^{s}\,.
\]
\item For every $k,l\in\N$, there exists $C>0$ such that for every $n\in\N$, $t\in\R$ and $\varepsilon>0$, 
\[\|v_n^{\varepsilon}(t)\|_{H^k_xH^l_y}\leq C \kappa_n n^{k+2l-2s}(1+(t\lambda_n^2)^{k+l})\min\left(1,\frac{1}{\varepsilon n}\right)^{k+2l+3}\,.\]
\item Moreover, with the same notation, we have 
\[\|\partial_x^k\partial_y^l v_n^{\varepsilon}(t)\|_{L^{\infty}}\leq C \lambda_n n^{k+2l}(1+(t\lambda_n^2)^{k+l})\min\left(1,\frac{1}{\varepsilon n}\right)^{k+2l+3}\,.
\]
\end{enumerate}
\end{lem}

Let us denote by $t\mapsto \Phi(t)(f)$ the local maximal solution to~\eqref{eq:NLSHW} with initial data $f$ in $\mathcal{H}^m$ when $m>\frac 12$.
In order to define a dense subset $S$ of the pathological set
\[
\mathcal{P}=\left\{f\in\mathcal{H}^s(\R^2)\mid \limsup_{\varepsilon\to0,t\to 0}\|\Phi(t)(\rho_{\varepsilon}\ast f)\|_{\mathcal{H}^s}\to\infty\right\}\,,
\] we apply the {\it ``tanghuru''} construction. The $k$-th bubble $v_{0,k}$ is centered at any point $y_k\in\R^2$, with scaling parameter $n_k=e^{a^k}$ for some $1\ll a$:
\[
v_{0,k}(x,y):=v_{n_k}(x,y-y_k)
	=\log(n_k)^{-\gamma}n_k^{3/2-2s}\varphi(n_kx,n_k^2(y-y_k))\,.
\]
Then for $k,l\geq k_0$, we have that $v_{0,k}$ and $v_{0,l}$ have disjoint supports with radius of respective orders $n_k^{-1}$ and $n_l^{-1}$, moreover,
\[
\dd(\operatorname{supp}(v_{0,k}),\operatorname{supp}(v_{0,l}))\sim \frac{1}{k}-\frac{1}{l}\,.
\]
Applying the convolution by the function $\rho_{\varepsilon_{n_k}}$ which has a support of size $\varepsilon_{n_k}^{-1}=100/n_k$, we know that
\[
\operatorname{supp}(\rho_{\varepsilon_{n_k}}\ast v_{0,k})\subset \left\{(x,y)\in\R^2\mid |y-y_k|\leq C e^{-a^k}\right\}\,.
\]

\begin{definition}[Dense subset of the pathological set]
We denote by $S$ the set of initial data $f$ which can be decomposed under the form
\[
f=u_0+\sum_{k=k_1}^{\infty}v_{0,k}\,,\quad \text{for}\ k_1\geq 1\ \text{and}\ u_0\in \mathcal{C}_c^{\infty}(\R^2)\,.
\]
\end{definition}
One can see that $S$ defines a dense subset of $\mathcal{H}^s$. We make use of the following precise upper bounds in $H^m$, $m\geq 0$, of the initial data regularized by convolution, see Lemma~2.4 in~\cite{CampsGassot2022}.
\begin{lem}[Norm inflation of the lollipop by convolution]
Let $m_1,m_2\in\N$, $k\geq k_0\geq 1$. If $m_1=m_2=0$, there holds
\[
\sum_{l=k_0}^{k-1}\|\rho_{\varepsilon_{n_k}}\ast v_{0,l}\|_{H^{m_1}_xH^{m_2}_y}
	\lesssim 1\,,\quad 
\sum_{l=k}^{\infty}\|\rho_{\varepsilon_{n_k}}\ast v_{0,l}\|_{H^{m_1}_xH^{m_2}_y}
	\lesssim n_{k+1}^{m_1+2m_2-s}\left(\frac{n_k}{n_{k+1}}\right)^{\frac 32}\,.
\]
If $m_1+2m_2\geq 1$, we have 
\[
\sum_{l=k_0}^{k-1}\|\rho_{\varepsilon_{n_k}}\ast v_{0,l}\|_{H^{m_1}_xH^{m_2}_y}
	\lesssim n_{k-1}^{m_1+2m_2-s}\,,\quad 
\sum_{l=k}^{\infty}\|\rho_{\varepsilon_{n_k}}\ast v_{0,l}\|_{H^{m_1}_xH^{m_2}_y}
	\lesssim n_{k}^{m_1+2m_2}n_{k+1}^{-s}\left(\frac{n_k}{n_{k+1}}\right)^{\frac 32}\,.
\]
\end{lem}

Note that this estimate is not valid for every $\varepsilon<\varepsilon_{n_k}$, but only when $\varepsilon\geq \varepsilon_{n_k}$. As a consequence, we do not have a uniform control on the time of existence of the smooth solution with initial data $\rho_{\varepsilon}\ast f$ in $L^2_xH^{m}_y$ for $m>1/2$.

It remains to study the time evolution of $\rho_{\varepsilon_{n_k}}\ast f$, by comparing it to the ODE profile $\rho_{\varepsilon_{n_k}}\ast v_{0,k}$ only. An adaptation of the proof of Proposition~2.6 in~\cite{CampsGassot2022} leads to the following result. Let $u^{\varepsilon_{n_k}}$ be the maximal solution of~\eqref{eq:NLSHW} with initial data $\rho_{\varepsilon_{n_k}}\ast f$ in $L^2_xH^1_y$, and $T_*(\varepsilon_{n_k})$ the maximal time of existence. Note that according to~\cite{BahriIbrahimKikuchi2020remarks}, Theorem 1.6, the Cauchy problem for~\eqref{eq:NLSHW} is locally well-posed in this space. Moreover, $T_*(\varepsilon_{n_k})$ only depends on the norm of the initial data in $L^2_xH^1_y$. Therefore, the norm of the solution must blow up when $t\to T_*(\varepsilon_{n_k})$. We establish that this does not happen before the norm inflation time $t_{n_k}$ by running a perturbative analysis.

\begin{prop}[Growth of perturbed initial data]
Fix $0<\gamma<\beta<1$. Let $0\leq s<\frac{1}{4}$. Let $n=n_k$. Then the solution $u^{\varepsilon_n}$ is well-defined in $\mathcal{C}([0,t_n],L^2_xH^{1}_y)$. Moreover, there holds
\[
\|u^{\varepsilon_n}(t_n)-\rho_{\varepsilon_n}\ast e^{it\A}u_0-v_n^{\varepsilon_n}(t_n)\|_{L^2_xH^{s}_y}\leq C,
\]
as a consequence,
\[
\|u^{\varepsilon_n}(t_n)\|_{L^2_xH^{s}_y}
	\gtrsim \log(n)^{2s(\beta-\gamma)-\gamma}.
\]
\end{prop}

As a consequence, the set $S$ is a subset of the pathological set $\mathcal{P}$.

\section*{Acknowledgments}
This material is based upon work supported by the National Science Foundation under Grant No. DMS-1929284 while the authors were in residence at the Institute for Computational and Experimental Research in Mathematics in Providence, RI, during the program “Hamiltonian Methods in Dispersive and Wave Evolution Equations”.
The authors would like to thank the organizers of the program, and all the staff at ICERM for their great hospitality. S. Ibrahim is supported by the NSERC grant No. 371637-2019.



\end{document}